\documentclass{amsart}
\usepackage[latin1]{inputenc}

\usepackage{amssymb}
\usepackage{amsmath,enumerate,rotate}
\usepackage{amssymb,latexsym}
\usepackage{epsfig}
\usepackage{graphicx}
\usepackage[english]{babel}
\usepackage{natbib}
\usepackage{color}

\usepackage{a4wide}

\newtheorem{definition}{Definition}[section]
\newtheorem{theorem}[definition]{Theorem}
\newtheorem{lemma}[definition]{Lemma}
\newtheorem{proposition}[definition]{Proposition}

\newtheorem{remark}{Remark}

\def \QQ{ {\mathcal{S}}}

\newcommand{\qn}{\hat{q}_n}
\newcommand{\qj}{\hat{q}_j}
\newcommand{\qnj}{\hat{q}_{n-j}}

\def\CF { \mathcal{F}}

\def\CB {\mathcal{B}}

\def\CI { \mathcal{I}}

\def\CI {\mathcal{I}}

\def\CG {\mathcal{G}}

\def§{\mathcal{S}}

\def\W {\mathcal{W}}
\def\b {\beta}
\def\a {\alpha}

\def\d {\delta}
\def þ{\theta}
\def\s {\sigma}

\def\g {\gamma}
\def\eps {\epsilon}
\def\n {\nu}
\def þ{\Theta}
\def\a {\alpha}
\def\g {\gamma}

\def\z {\zeta}
\def\lm {\lambda}

\defÞ{\Theta}

\def\J {\mathbb{I}}

\def\E {\mathbb{E}}
\def\RE {\mathbb{R}}

\def\setN {\mathbb{N}}

\renewcommand{\Re}{\mathbb{R}}

\def \fh {\widehat{h}}

\usepackage{verbatim}

\begin{document}

\title{Central limit theorem for a class of one-dimensional kinetic equations}

\author{Federico Bassetti}

\email{federico.bassetti@unipv.it}

\address{Università degli Studi di Pavia \\
via Ferrata 1, 27100, Pavia, Italy}

\author{Lucia Ladelli}

\email{lucia.ladelli@polimi.it}

\address{Politecnico di Milano\\
        piazza Leonardo da Vinci 32, 20133, Milano, Italy \\
        {\rm also affiliated to CNR-IMATI Milano, Italy}}

\author{Daniel Matthes}
\email{matthes@asc.tuwien.ac.at}
\address{Institut für Analysis und Scientific Computing \\
           Technische Universität Wien \\
        Wiedner Hauptstraße 8--10 \\
           1040 Wien, Austria}


\begin{abstract}
  We introduce a class of Boltzmann equations on the real line,
  which constitute extensions of the classical Kac caricature.
  The collisional gain operators are defined by smoothing transformations 
  with quite general properties.
  By establishing a connection to the central limit problem,
  we are able to prove long-time convergence of the equation's 
  solutions towards a limit distribution.
  If the initial condition for the Boltzmann equation
  belongs to the domain of normal attraction of a certain stable law $\nu_\alpha$,
  then the limit is a scale mixture of $\nu_\alpha$.
  Under some additional assumptions, 
  explicit exponential rates for the convergence to equilibrium in Wasserstein metrics are calculated,
  and strong convergence of the probability densities is shown.
\end{abstract}

\maketitle

\section{Introduction}

In a variety of recent publications, 
intimate relations between the central limit theorem of probability theory
and the celebrated Kac caricature of the Boltzmann equation from statistical physics
have been revealed.
The idea to represent the solutions of the Kac equation in a probabilistic way
dates back at least to the works of McKean in the 60's, see e.g. 
\cite{McKean1966},
but has been fully formalized and employed in the derivation of 
analytic results 
only in the last decade.
For instance, probabilistic methods have been used to get 
estimates on the quality of approximation of solutions by truncated 
Wild sums in \cite{CarlenCarvalhoGabetta2000}, to 
study necessary and sufficient 
conditions for the convergence to a steady state in
\cite{GabettaRegazziniCLT}, to study the blow-up behavior of solutions of infinite energy 
in \cite{Carlen2007,CarlenGabettaRegazzini2008}, 
to obtain rates of convergence to equilibrium 
of the solutions both in strong and weak metrics, 
\cite{GabettaRegazziniWM,dolera,dolera2}.
The power of the probabilistic approach is illustrated, for instance, 
by the fact that in \cite{dolera} very refined estimates 
for the classical central limit theorem
enabled the authors to deliver the first proof of a conjecture 
that has been formulated by McKean about fourty years ago. 

The applicability of probabilistic methods is not restricted 
to the classical Kac equation,
but extends to the inelastic Kac model,
proposed by \cite{PulvirentiToscani}.
In the inelastic model, the energy (second moment) of the 
solution is not conserved but dissipated,
and hence infinite energy is needed initially to obtain a non-trivial long-time limit.
In  \cite{bassettiladelliregazzini} probabilistic methods have been
used to study the speed of approach to equilibrium under the assumption
that the initial condition belongs to the domain of normal 
attraction of a suitable stable law. In this context, indeed, the steady states
are the corresponding stable laws. 

In the current paper, we continue in the spirit of the aforementioned results.
By means of  the central limit theorem for triangular arrays,
we are able to study the long time behavior  of solutions 
of a wide class of one-dimensional Boltzmann equations,
which contains (essentially) the classical and 
the inelastic Kac model as special cases.

To be more specific,
recall that the Kac equation describes the evolution of a time-dependent
probability measure $\mu(t)$ on the real axis,
and is most conveniently written as an evolution equation for the characteristic function $\phi(t)$ of $\mu(t)$.
The equation has the form
\begin{equation}
  \label{eq.boltzivp}
  \left\{ \begin{aligned}
      & \partial_t\phi(t;\xi)+\phi(t;\xi) = 
      \widehat{Q}^+[\phi(t;\cdot),\phi(t;\cdot)](\xi) 
       \qquad(t>0, \xi \in  \RE)\\
      & \phi(0;\xi)=\phi_0(\xi) \\
    \end{aligned}
  \right.
\end{equation}
where the collisional gain operator is given by
\begin{align}
  \label{eq.collop}
  \widehat{Q}^+[\phi(t;\cdot),\phi(t;\cdot)](\xi) 
  := \E[\phi(t;L\xi)\phi(t;R\xi)].
\end{align}
Above, $(L,R)$ is a random vector defined on a probability 
space $(\Omega,\CF,P)$ and $\E$ denotes the expectation 
with respect to $P$. 
The initial condition $\phi_0$ is the characteristic function of a prescribed 
real random variable $X_0$; by abuse of notation, 
we shall also refer to $X_0$, to its probability distribution function $F_0$ 
or to its law $\mu(0)$ as the initial condition.

For the classical Kac equation, one writes
$(L,R)=(\sin(Þ),\cos(Þ))$, with
$Þ$ uniformly distributed on $[0,2\pi)$, 
and hence $L^2+R^2=1$ a.s.
The inelastic Kac equation is obtained by
\[
(L,R)=(\sin(Þ)|\sin(Þ)|^p,\cos(Þ)|\cos(Þ)|^p),
\] 
with $p>0$, and hence
$|L|^\a+|R|^\a=1$ a.s., { if} $\a=2/(1+p)$.
%
%
It is worth recalling that the study of the respective initial value problems 
can be reduced to the study of the same problems under the additional assumption 
that the initial distribution is symmetric, 
i.e. the initial 
characteristic function is real, and $(L,R)=(|\sin(Þ)|^{1+p}, 
|\cos(Þ)|^{1+p})$. See Section \ref{sct.examples}.

In this paper, we consider the problem \eqref{eq.boltzivp},
where the random { variables} $L$ and $R$ 
in the definition of the collision operator in \eqref{eq.collop} 
are non-negative and satisfy the condition
\begin{align}
  \label{eq.alpha}
  \E[L^\alpha+R^\alpha] = 1 ,
\end{align}
for some $\a$ in $(0,2]$.
The therewith defined bilinear operators $\widehat{Q}^+$
are examples of smoothing transformation, 
which have been extensively studied in the context of branched random walks,
see e.g. \cite{Kahane1976,DurrettLiggett1983,Guivarch1990,Liu1998,Iksanov} and the 
references therein.

Our motivation, however,  originates from applications to statistical physics.
These applications are discussed in Section \ref{sct.examples}.
At this point, we just mention the two main examples.
\begin{enumerate}
\item 
  Passing from the Kac condition $L^2+R^2=1$ to \eqref{eq.alpha} with $\a=2$,
  the model retains its crucial physical property to conserve the 
  second moment of the solution.
  However, the variety of possible steady states grows considerably:
  depending on the law of $(L,R)$, the latter may exhibit heavy tails.
\item
  For certain distributions satisfying \eqref{eq.alpha} with $\a=1$,
  equation \eqref{eq.boltzivp} has been used to model the redistribution of wealth in simplified market economies,
  which conserve the society's total wealth (first moment).
  Whereas the condition $L+R=1$ would correspond to deterministic trading 
  and lead eventually to a fair but unrealistic distribution of wealth 
  in the long time limit,
  the relaxed condition \eqref{eq.alpha} allows trade mechanisms that involve randomness
  (corresponding to risky investments)
  and lead to a realistic, highly unequal distribution of wealth.
\end{enumerate}


Our main results from Theorems \ref{thm2} and \ref{thm3-bis} can be rephrased as follows:

\vskip 0.5cm
%
 {\it  Assume that \eqref{eq.alpha} holds with $\a\in(0,2]$, but $\a\neq1$,
  and in addition that $\E[L^\gamma+R^\gamma]<1$ for some $\gamma>\a$.
  Let $\mu(t)$, for $t\geq0$, be the probability measure on $\RE$ 
  such that its characteristic function $\phi(t)$ is the unique 
   solution to the associated Boltzmann equation \eqref{eq.boltzivp}.
  Assume further that the initial datum $\mu(0)$ lies in the normal domain of attraction of some $\alpha$-stable law $\nu_\alpha$,
  and that $\mu(0)$ is centered if $\a>1$.
  Then, as $t\to + \infty$, the probability  measures $\mu(t)$ converge weakly to a 
   limit distribution $\mu_\infty$,
  which is a non-trivial scale mixture of  $\nu_\alpha$.
}
%

\vskip 0.5cm

The results in the case $\a=1$ are more involved;
see Theorems \ref{thm1} and \ref{thm2-bis}.

Under the previous  general hypotheses,
no more than weak convergence can be expected.
However, slightly more restrictive assumptions on the initial condition $\phi_0$ suffice
to obtain exponentially fast convergence in some Wasserstein distance.
Finally, if the initial condition possesses a density with finite {Linnik-Fisher functional}  
and the condition $L^r+R^r\geq1$ holds a.s. for some $r>0$,
then the probability density of $\mu(t)$ exists for every $t>0$ and converges strongly 
in the Lebesgue spaces $L^1(\RE)$ and $L^2(\RE)$ as $t \to +\infty$.

The largest part of the paper deals with the proofs of weak convergence,
which are obtained in application of the central limit theorem for triangular arrays.
Consequently, the core element of the proof is 
to establish a suitable probabilistic interpretation 
of the solution to \eqref{eq.boltzivp}.
The link to probability theory is provided by a semi-explicit solution formula:
the Wild sum,
\begin{align}
  \label{eq.wild0}
  \phi(t) = e^{-t} \sum_{n=0}^\infty (1-e^{-t})^n \qn ,
\end{align}
represents the solution $\phi(t)$ as a convex combination of characteristic functions $\qn$,
which are obtained by iterated application of the gain operator $\widehat{Q}^+$ 
to the initial condition $\phi_0$ --- see formula \eqref{eq.wildrec}. 
Following \cite{GabettaRegazziniCLT} we consider a sequence of random variables $W_n$ 
such that $W_n$ has $\hat q_{n-1}$ as its characteristic function,
and possesses the representation
\begin{align}
  \label{eq.defw}
  W_n = \sum_{j=1}^n \b_{j,n}X_j ,
\end{align}
where the $X_j$ are independent and identically distributed random variables 
with common characteristic function $\phi_0$.
The weights $\b_{j,n}$ are random variables themselves
and are obtained in a recursive way, see \eqref{recursion}.

The behavior of $\phi(t)$ in \eqref{eq.wild0} as $t\to\infty$ is obviously determined
by the behavior of the law of $W_n$ as $n\to\infty$.
It is important to note that a direct application of the central limit theorem 
to the study of $W_n$ 
is inadmissible since the weights in (\ref{eq.defw}) are not independent.
However, one can apply the central limit theorem
to study  the conditional law of $W_n$, given the array 
of weights $\b_{j,n}$.
 

Representations in the form \eqref{eq.wild0} with \eqref{eq.defw} 
are known for the (classical and inelastic) Kac equation, 
see \cite{GabettaRegazziniCLT} and \cite{bassettiladelliregazzini}.
The situation here is more involved, 
since \eqref{eq.alpha} only implies that
\begin{align}
  \label{eq.rich}
  \E\big[\b_{1,n}^\a+\b_{2,2}^\a+\cdots+\b_{n,n}^\a\big] = 1,
\end{align}
whereas for the Kac equation,
\begin{align}
  \label{eq.poor}
  \b_{1,n}^\a+\b_{2,n}^\a+\cdots+\b_{n,n}^\a = 1 \quad a.s.
\end{align}
In order to be able to apply the central limit theorem,
one needs to prove that $\max_{1 \leq j \leq n}|\b_{j,n}|$ converges in probability to zero,
and that $\sum_{j} \b_{j,n}^\a$ converges (almost surely) to a random variable. 
Thanks to (\ref{eq.poor}), 
the latter condition is immediately satisfied for the Kac equation, 
while it is not always true for the general model considered here. 
We stress that the generality of \eqref{eq.rich} in comparison to \eqref{eq.poor} 
is the origin of the richness of possible steady states in \eqref{eq.boltzivp}.

The paper is organized as follows.
In Section \ref{sct.boltzmann}, 
we recall some basic facts about the Boltzmann equation under consideration,
present a couple of examples to which the theory applies,
and derive the stochastic representation of solutions.
Section \ref{sct.results} contains the statements of our main theorems.
The results are classified into those on convergence in distribution (Section \ref{sct.weak}),
convergence in Wasserstein metrics at quantitative rates (Section \ref{sct.wasserstein})
and strong convergence of the probability densities (Section \ref{sct.strong}).
All proofs are collected in Section \ref{sct.proofs}.

\section{Examples and preliminary results}
\label{sct.boltzmann}


One-dimensional kinetic equations of type \eqref{eq.boltzivp}-\eqref{eq.collop}, 
like the Kac equation and its variants, 
provide simplified models for a spatially homogeneous gas,
in which particle move only in one spatial direction.
The measure $\mu(t)$, whose characteristic function is the solution of (\ref{eq.boltzivp}), 
describes the probability distribution of the velocity of a molecule at time $t$. 
%
%
The basic assumption is that particles 
change their velocities only because of binary collisions.
When two particles collide, 
then their velocities change from $v$ and $w$, respectively, to
\begin{align}
  \label{eq.collision}
  v' = L_1 v + R_1 w \quad\mbox{and}\quad w' = R_2 v + L_2 w 
\end{align}
with $L_1=L_2=\sin(Þ)$ and $R_1=-R_2=\cos(Þ)$.

More generally, one can consider binary interaction obeying (\ref{eq.collision}),
where $(L_1,R_1)$ and $(L_2,R_2)$ are two identically distributed random vectors 
(not necessarily independent)
with the same law of $(L,R)$.
This leads, at least formally, to equation (\ref{eq.boltzivp}).

\subsection{Examples}
\label{sct.examples} 

The following applications are supposed to serve to motivate the study of
the Boltzmann equation \eqref{eq.boltzivp} with the condition \eqref{eq.alpha}.
The first two examples are taken from gas dynamics,
while the third originates from econophysics.

\subsubsection*{Kac like models} 
Instead of discussing the physical relevance of the Kac model, 
we simply remark that it constitutes the most sensible one-dimensional caricature 
of the Boltzmann equation for elastic Maxwell molecules in three dimensions.
A comprehensive review on the mathematical theory of the latter
is found e.g. in \cite{bobylev}.
The term ``elastically'' refers to the fact that 
the kinetic energy of two interacting molecules
-- which is proportional to the square of the particles' velocities --
is preserved in their collisions.
Indeed, since $L_1=L_2=\sin(Þ)$ and $R_1=-R_2=\cos(Þ)$,
one obtains $(v')^2+(w')^2=v^2+w^2$.

We shall not detail any of the numerous results 
available in the extensive literature on the Kac equation,
but simply summarize some basic properties that are connected with our investigations here.
First, we remark that
the microscopic conservation of the particles' kinetic energy 
implies the conservation of the average energy, 
which is the second moment of the solution $\mu(t)$ to \eqref{eq.boltzivp}.
Moreover, 
it is easily proven that the average velocity, i.e. the first moment of $\mu(t)$, 
converges to zero exponentially fast.
For $t\to +\infty$, the solution $\mu(t)$ converges weakly to a Gaussian measure
that is determined by the conserved  second moment.

As already mentioned in the introduction, the study of the original Kac model
can be reduced to the study of a particular case of the model we are considering. 
Indeed, it is well-known that the solution of the Kac equation can be written
as
\begin{align}
\phi(t,\xi)= e^{-t} Im \big(\phi_0(\xi)\big) + \phi^*(t,\xi)
\end{align}
where $\phi^*$ is the solution to problem (\ref{eq.boltzivp}) with $Re(\phi_0)$ in the place of $\phi_0$, $L=|\sin(Þ)|$ and  $R=|\cos(Þ)|$. 
Hence, 
we can invoke Theorem \ref{thm3-bis},
which provides another proof of the large-time convergence of solutions 
$\mu(t)$ to a Gaussian law.
In fact, also Theorem \ref{thm.strongnew} is applicable,
which shows that the densities of $\mu(t)$ converge in $L^1(\Re)$ and $L^2(\Re)$,
provided $\mu(0)$ possesses a density with finite Linnik functional.

These consequences are weak in comparison to the various extremely refined
convergence estimates for the solutions to the Kac equation available in the literature. 
See, e.g., the review \cite{RegazziniUMI}.
On the other hand, our proofs do not rely on any of the symmetry properties
that are specific for the Kac model.
Thus, our aforementioned results extend --- word by word --- 
to the wide class of problems \eqref{eq.boltzivp}-\eqref{eq.collop} with $L^2+R^2=1$ a.s.

The variant of the Kac equation, introduced in \cite{PulvirentiToscani},
is called inelastic because the total kinetic energy 
of two colliding particles is not preserved in the collision mechanism,
but decreases in general.
Consequently, if the second moment of the initial condition $\mu(0)$ is finite,
then the second moment of the solution $\mu(t)$ converges to zero exponentially fast in $t>0$.
Non-trivial long-time limits are thus necessarily obtained 
from initial conditions with infinite energy.
In \cite{bassettiladelliregazzini}, it is shown that
the solution $\mu(t)$ converges weakly to a $\a$-stable law $\nu_\alpha$
if $\mu(0)$ belongs to the normal domain of attraction of $\nu_\alpha$,

As for the Kac model, the study of the inelastic  Kac model too 
can be reduced to the framework of the present paper. 
Hence, Theorem \ref{thm.strongnew} and Proposition \ref{PropW-2}
yield new results concerning strong convergence of densities in $L^1(\RE)$ and
convergence with respect to the Wasserstein metrics.

\subsubsection*{Inelastic Maxwell molecules}
We shall now consider a variant of the Kac model in which
the energy is {\em not} conserved in the individual particle collisions,
but gains and losses balance in such a way that 
the average kinetic energy is conserved.
This is achieved by relaxing the condition $L^2+R^2=1$ to $\E[L^2+R^2]=1$, 
which is \eqref{eq.alpha} with $\a=2$.

Just as the Kac equation is a caricature of the Boltzmann equation for elastic Maxwell molecules,
the model at hand can be thought of as a caricature 
of a Boltzmann equation for {\em inelastic} Maxwell molecules in three dimensions.
For the definition of the corresponding model, its physical justification, 
and a collection of relevant references, see \cite{carrillo}.
We stress, however, that the Kac caricature of inelastic Maxwell molecules
is not the same as the inelastic Kac model from the preceeding paragraph. 

Conservation of the total energy can be proven  for centered solution $\mu(t)$;
like symmetry, also centering is propagated from $\mu(0)$ to any $\mu(t)$ by \eqref{eq.boltzivp}.
The argument leading to energy conservation is given in the remarks following Theorem \ref{thm3-bis}.

Relaxation from strict energy conservation to conservation in the mean
affects the possibilities for the large-time dynamics of $\mu(t)$.
It follows from Theorem \ref{thm3-bis} that if $\E[L^\gamma+R^\gamma]<1$ for some $\gamma>2$, 
then any solution $\mu(t)$,
which is centered and of finite second moment initially,
converges weakly to a non-trivial steady state $\mu_\infty$.
However, unless $L^2+R^2=1$ a.s., $\mu_\infty$ is {\em not} a Gaussian.
In fact, $(L,R)$ can be chosen in such a way that $\mu_\infty$
possesses only a finite number of moments.
In physics, 
such velocity distributions are referred to as ``high energy tailed'', 
and typically appear when the molecular gas is connected to a thermal bath.

An example leading to high energy tails is the following:
let $(L,R)$ such that $P\{L=1/2\}=P\{L=\sqrt{5}/2\}=1/2$ and $P\{R=1/2\}=1$.
One verifies that $\E[L^2+R^2]=1$ and $\E[L^4+R^4]=7/8<1$,
so Theorem \ref{thm3-bis} guarantees the existence of a non-degenerate steady state $\mu_\infty$.
Moreover, $\E[L^6+R^6]=1$, 
and one concludes further from Theorem \ref{thm3-bis}
that the sixth moment of $\mu_\infty$ diverges, whereas all lower moments are finite.

\subsubsection*{Wealth distribution}
Recently, 
an alternative interpretation of the equation \eqref{eq.boltzivp} 
has become popular.
The homogeneous gas of colliding molecules is replaced 
by a simple market with a large number of interacting agents.
The current ``state'' of each individual is characterized by a single number, 
his or her wealth $v$.
Correspondingly, the measure $\mu(t)$ represents the distribution of wealth among the agents.
The collision rule \eqref{eq.collision} describes how wealth is exchanged between agents
in binary trade interactions. See, e.g., \cite{slanina,PareschiToscani}.

Typically,
it is  assumed that $\mu(t)$ is supported on the positive semi-axis.
In fact, the first moment of $\mu(t)$ represents the total wealth of the society
and plays the same rôle as the energy in the previous discussion.
In particular, it is conserved by the evolution.

In the first approaches, see e.g. \cite{angle},
conservation of wealth in each trade was required,
i.e. $v'+w'=v+w$. Hence, assuming $L_1=L_2$ and $R_1=R_2$ 
in \eqref{eq.collision}, this yields $L+R=1$ a.s.
However, the obtained results were unsatisfactory:
in the long time limit, 
the wealth distribution $\mu(t)$ concentrates on the society's average wealth,
so that asymptotically, all agents possess the same amount of money.
This also follows from our Theorem \ref{thm1}.

More realistic results have been obtained by \cite{MatthesToscani}, 
where trade rules $(L,R)$ have been introduced 
that satisfy \eqref{eq.alpha} with $\a=1$, but in general $P\{L+R =1\}<1$.
Thus wealth can be increased or diminished in individual trades,
but the society's total wealth, i.e. the first moment of $\mu(t)$, remains constant in time.
The proof of conservation of the mean wealth can also be found 
in the remarks after Theorem \ref{thm1}.

A typical example for trade rules is the following.
Let $L_1=L_2$ and $R_1=R_2$ with $P\{L=1-p+r\}=P\{L=1-p-r\}=1/2$ and $P\{R=p\}=1$,
where $p$ in $(0,1)$ is a relative price of an investment and $r$ in $(0,p)$ is a risk factor.
The interpretation reads as follows:
each of the two interacting agents buys from the other one some risky asset 
at the price of the $p$th fraction of the respective buyer's current wealth;
these investments either pay off and produce some additional wealth, or lose value,
both proportional (with $r$) to their original price.
Over-simplified as this model might be, 
it is able to produce (for suitable choices of $p$ and $r$) steady distributions $\mu_\infty$ 
with only finitely many moments, 
as are typical wealth distributions for western countries; 
see \cite{MatthesToscani} for further discussion.

An example is provided by choosing $p=1/4$ and $r=1/2$.
One easily verifies that $\E[L+R]=1$, $\E[L^2+R^2]=7/8<1$ and $\E[L^3+R^3]=1$.
By Theorem \ref{thm1}, 
it follows that there exists a non-degenerate steady distribution $\mu_\infty$
that possesses all moments up to the third, 
whereas the third moment diverges.

\subsection{Probabilistic representation of the solution}

{As already mentioned, a convenient way to represent the solution $\phi$ to the problem \eqref{eq.boltzivp}
is
\begin{equation}
  \label{eq.wildsum}
  \phi(t;\xi)=\sum_{n=0}^\infty e^{-t}(1-e^{-t})^{n}\qn(\xi) \qquad(t\geq0, \xi \in \RE)
\end{equation}
where $\qn$ is recursively defined by
\begin{equation}
  \label{eq.wildrec}
  \left \{
    \begin{array}{ll}
      \hat q_0(\xi):=\phi_0(\xi)& \quad \\
      \hat q_{n}(\xi):= \frac{1}{n} \sum_{j=0}^{n-1} \E[\hat q_j(L\xi)\hat q_{n-1-j}(R\xi)]
      & \qquad(n=1,2,\dots).\\
    \end{array}
  \right .
\end{equation}
The series in \eqref{eq.wildsum} is referred to as {\em Wild sum}, 
since the representation (\ref{eq.wildsum}) has been derived in \cite{Wild1951} for the solution of the
 Kac equation. }
In this section, we shall rephrase  the Wild sum in a probabilistic way. 
The idea goes back to \cite{McKean1966,McKean1967},
where McKean relates the Wild series to a random walk on a
class of binary trees, the so--called McKean trees. 
%
It is not hard to verify that each of the expressions $\qn$ in the Wild series
is indeed a characteristic function.
Now, following \cite{GabettaRegazziniCLT}, we shall 
define a sequence of random variables $W_n$ 
such that $\hat q_{n-1}(\xi)=\E[e^{i\xi W_n}]$.

On a sufficiently large probability space $(\Omega,\CF,P)$,
let the following be given:
\begin{itemize}
\item a sequence $(X_n)_{n\in\setN}$ of independent and identically distributed random variables 
with common distribution function  $F_0$;
\item a sequence $\big((L_n,R_n)\big)_{n\in\setN}$ of independent and identically distributed random 
vectors, distributed as $(L,R)$;
\item a sequence $(I_n)_{n\in\setN}$ of independent integer random variables,
  where each $I_n$ is uniformly distributed on the indices $\{1,2,\ldots,n\}$;
\item a stochastic process $(\nu_t)_{t\geq0}$ with values in $\setN$ 
 and $P\{\nu_t=n\}=e^{-t}(1-e^{-t})^{n-1}$.
\end{itemize}
We assume further that
\[ (I_n)_{n\geq 1}, \quad (L_n,R_n)_{n\geq 1}, \quad (X_n)_{n\geq 1} \quad\mbox{and} \quad (\nu_t)_{ t>0 } \]
are stochastically independent.
The random array of weights $[\b_{j,n}: j=1,\dots,n]_{n \geq 1}$ is recursively defined as follows:
\[
\beta_{1,1}:=1, \qquad (\beta_{1,2},\beta_{2,2}):=(L_1,R_1)
\]
and, for any $n\geq 2$,
\begin{equation}
  \label{recursion}
  (\beta_{1,n+1},\ldots,\beta_{n+1,n+1}) := (\beta_{1,n},\ldots,\beta_{I_n-1,n}, L_n \beta_{I_n,n}, R_n \beta_{I_n,n},\beta_{I_n+1,n},\ldots, \beta_{n,n}). 
\end{equation}
Finally set
\begin{align}
  \label{eq.defv}
  W_n:=\sum_{j=1}^n \beta_{j,n} X_j \quad
  \text{and} \quad
  V_t:=W_{\nu_t}=\sum_{j=1}^{\n_t} \beta_{j,\n_t} X_j.
\end{align}

\begin{figure}[ht]
  \centering
  \includegraphics [scale=0.6]{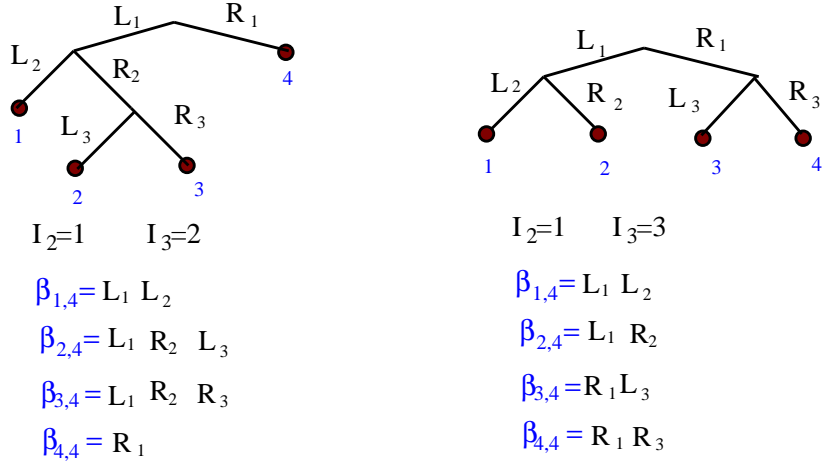}
  \caption{Two McKean trees, with associated weights $\b$.}
  \label{FigureA}
\end{figure}

There is a direct interpretation of this construction in terms 
of McKean trees.
For an introduction to McKean trees, see, e.g., \cite{CarlenCarvalhoGabetta2000}.
Each finite sequence $\CI_n=(I_1,I_2,\ldots,I_{n-1})$ corresponds to a McKean tree with $n$ leaves.
The tree associated to $\CI_{n+1}$ is obtained from the tree associated to $\CI_n$ upon
replacing the $I_n$-th leaf (counting from the left) by a binary branching with two new leaves.
The left of the new branches is labelled with $L_n$, and the right one with $R_n$.
Finally, the weights $\b_{j,n}$ are associated to the leaves of the $\CI_n$-tree;
namely, $\b_{j,n}$ is the product of the labels assigned to the branches 
along the ascending path connecting the $j$th leaf to the root.
The trees with $\CI_4=(1,1,2)$ and $\CI_4=(1,2,3)$, respectively, are 
displayed in Figure \ref{FigureA}.

In the Wild construction \eqref{eq.wildrec}, 
McKean trees with $n$ leaves are obtained by joining pairs of trees with $k$ and $n-k$ leaves,
respectively, at a new common root.
Wheras our construction produces the $n$ leaved trees from the $n-1$ leaved trees
replacing a leaf by a binary branching.
In a way, the second construction is much more natural --- or, at least, more biological!
The next proposition shows that both constructions indeed lead to 
the same result.

In the rest of the paper expectations with respect to $P$ 
will be denoted by $\E$.

\begin{proposition}[Probabilistic representation]\label{Prop:probint} 
Equation \eqref{eq.boltzivp} has a unique solution $\phi(t)$, which   
  coincides with the characteristic function of $V_t$, i.e.
  \begin{align*}
    \phi(t,\xi) = \E[e^{i \xi V_t} ] = \sum_{n=0}^\infty e^{-t}(1-e^{-t})^{n} \E[e^{i \xi W_{n+1}} ]
    \qquad (t >0, \, \xi \in \RE).
  \end{align*}
\end{proposition}

\begin{proof} 
  The respective proof for the Kac case is essentially already 
  contained in \cite{McKean1966}.
  See   \cite{GabettaRegazziniCLT} for a more complete proof. 
  Here, we extend the argument to the problem \eqref{eq.boltzivp}.
  First of all it is easy to prove, following \cite{Wild1951} and \cite{McKean1966},
   that formulas \eqref{eq.wildsum} and (\ref{eq.wildrec})
  produce the unique solution to problem \eqref{eq.boltzivp}. See also  \cite{Sznitman}.
  Hence, comparing the Wild sum representation \eqref{eq.wildsum}
  and the definition of $V_t$ in \eqref{eq.defv}, 
  it obviously suffices to prove that
  \begin{equation}
    \label{ind}
    \hat q_{\ell-1}(\xi)= \E[e^{i \xi W_\ell} ],
  \end{equation}
  which we will show by induction on $\ell\geq1$.
  First, note that $ \E [\exp(i \xi W_1) ]=\E [\exp(i\xi X_1)]=\phi_0(\xi)=\hat q_0(\xi)$ and
  \begin{align*}
    \E [e^{i \xi W_2} ]  =   
    \E [e^{i\xi(L_1 X_1+R_1 X_2)}] =   \E [\E [e^{i\xi(L_1 X_1+R_1X_2)}|L_1,R_1]]=\hat q_{1}(\xi),
  \end{align*}
  which shows \eqref{ind} for $\ell=1$ and $\ell=2$.
  Let $n\geq3$, and assume that \eqref{ind} holds for all $1\leq\ell<n$;
  we prove \eqref{ind} for $\ell=n$.

  Recall that the weights $\b_{j,n}$ are products of random variables $L_i$ and $R_i$.
  By the recursive definition in \eqref{recursion},
  one can define a random index $K_n<n$ such that
  all products $\b_{j,n}$ with $j \leq K_n$ contain $L_1$ as a factor, 
  while the remaining products $\b_{j,n}$ with $K_n+1\leq j \leq n$ 
   contain $R_1$.
  (In terms of McKean trees, 
  $K_n$ is the number of leaves in the left sub-tree, 
  and $n-K_n$ the number of leaves in the right one.)
  By induction it is easy to see that
  \[
  P\{K_n=i\}=\frac{1}{n-1} \qquad i=1,\dots,n-1;
  \]
  c.f. Lemma 2.1 in \cite{CarlenCarvalhoGabetta2000}.
  Now, 
  \[
  A_{K_n} := \sum_{j=1}^{K_n}\frac{\b_{j,n}}{L_1} X_j, 
  \quad  
  B_{K_n} := \sum_{j=K_n+1}^{n}\frac{\b_{j,n}}{R_1} X_j
  \quad \text{and} \quad (L_1,R_1)
  \] 
  are conditionally independent given $K_n$. 
  By the recursive definition of the weights $\b_{j,n}$ in \eqref{recursion},
  the following is easily deduced:
  the conditional distribution of $A_{K_n}$, given $\{K_n=k\}$,
  is the same as the (unconditional) distribution of $\sum_{j=1}^{k}\b_{j,k} X_j$,
  which clearly is the same distribution as that of $W_k$.
  Analogously, the conditional distribution of $B_{K_n}$, given  $\{K_n=k\}$, 
  equals the distribution of $\sum_{j=1}^{n-k}\b_{j,n-k}X_j$,
  which further equals the distribution of $W_{n-k}$.
  Hence, 
  \begin{align*}
    \E[e^{i\xi W_n}] &= \frac1{n-1} \sum_{k=1}^{n-1} \E\big[e^{i\xi(L_1A_k+R_1B_k)}\big|\{K_n=k\}\big] \\
    &= \frac1{n-1} \sum_{k=1}^{n-1} \E\big[\E[e^{i\xi L_1W_k}|L_1,R_1]\E[e^{i\xi R_1W_{n-k}}|L_1,R_1]\big] \\
    &= \frac1{n-1} \sum_{k=1}^{n-1} \E[\hat q_{k-1}(L_1\xi)\hat q_{n-k-1}(R_1\xi)] 
    = \frac1{n-1} \sum_{j=0}^{n-2} \E[\hat q_{n-2-j}(L_1\xi)\hat q_{j}(R_1\xi)]
   \end{align*}
  which is $\hat q_{n-1}$ by the recursive definition in \eqref{eq.wildrec}.
\end{proof}

\section{Convergence results}
\label{sct.results}

In order to state our results
we need to review some elementary facts about the central limit 
theorem for stable distributions.
Let us recall that a probability distribution is said to be 
{\em a centered stable law} of exponent $\a$ (with $0 < \a \leq 2$) if 
its characteristic function is of the form
\begin{equation}
  \label{chaSta}
  \hat g_\a(\xi)= 
  \left \{
    \begin{array}{ll}
      \exp\{ -k |\xi|^\a 
      (1-i \eta \tan(\pi\alpha/2)\operatorname{sign}\xi)   \} &  \text{if $\a \in (0,1) \cup (1,2)$}  \\
      \exp\{ -k |\xi| 
      (1+2i\eta/\pi \log|\xi| \operatorname{sign}\xi)   \} &  \text{if $\a=1$}  \\
      \exp\{ - \s^2 |\xi|^2/2 \} & \text{if $\a=2$.}
    \end{array}
  \right .
\end{equation}
where $k>0$ and $|\eta| \leq 1$.

By definition, a distribution function $F$  
belongs to the {\em domain of normal attraction} of 
a stable law of exponent $\a$ 
if for any sequence of independent and identically  distributed real-valued 
random variables $(X_n)_{n \geq 1}$ 
with common distribution function $F$, 
there exists a sequence of real numbers $(c_n)_{n \geq 1}$
such that the law of
\[
\frac{1}{n^{1/\a}} \sum_{i=1}^n X_i -c_n
\]
converges weakly to a stable law of exponent $\a \in (0,2]$.

It is well-known  that, provided $\a\not=2$,
$F$ belong to the domain of normal attraction of an $\a$-stable law  
if and only if $F$ satisfies 
\begin{equation}
  \label{stabledomain}
  \begin{split}
    \lim_{x \to +\infty} x^\a (1-F(x)) &=c^+<+\infty, \\
    \lim_{x \to -\infty} |x|^\a F(x) &=c^-<+\infty. \\
  \end{split}
\end{equation}
Typically, one also requires that $c^++c^->0$ in order to exclude 
convergence to the probability measure concentrated in $x=0$,
but here we shall include the situation $c^+=c^-=0$ as a special case.
The parameters $k$ and $\eta$ of the associated stable law 
in \eqref{chaSta}
are identified from $c^+$ and $c^-$ by
\begin{equation}
  \label{constant}
  k = (c^{+}+c^{-}) \frac{\pi}{2\Gamma(\a)\sin(\pi\a/2)}, 
  \qquad \eta = \frac{c^{+}-c^{-}  }{c^{+}+c^{-}}, 
\end{equation}
with the convention that $\eta=0$ if $c^++c^-=0$.
In contrast, if $\a=2$, 
$F$ belongs to the domain of normal attraction of a Gaussian law 
if and only if it has finite variance $\s^2$.

For more information on stable laws and central limit theorem see, 
for example, Chapter 2 of \cite{IbragimovLinnik1971} and 
Chapter 17 of \cite{fristedgray}.

\subsection{Convergence in distribution}
\label{sct.weak}

We return to our investigation of solutions to the initial value problem 
\eqref{eq.boltzivp}-\eqref{eq.collop}. 
For definiteness, let the two non-negative random variables $L$ and $R$,
which define the dynamics in \eqref{eq.collop}, be fixed from now on.
We assume that they satisfy
\begin{align}
  \label{eq.defalpha}
  \E[L^\a+R^\a] = 1 
\end{align}
for some number $\a\in(0,2]$.
We introduce the convex function $\QQ:[0,\infty)\to[-1,\infty]$ by
\begin{align*}
  \QQ (s)=\E[L^s+R^s]-1,
\end{align*}
where we adopt the convention $0^0=0$.
From \eqref{eq.defalpha} it follows that $\QQ(\a)=0$.
Recall that $F_0$ is the probability distribution function 
of the initial condition $X_0$ for \eqref{eq.boltzivp}, 
and its characteristic function is $\phi_0$.

The main results presented below 
show that if $F_0$ belongs to the domain of normal attraction of an $\alpha$-stable law,
then the solution $\phi(t;\cdot)$ to the problem \eqref{eq.boltzivp}-\eqref{eq.collop}
converges, as $t \to +\infty$,  
to the characteristic function of a mixture of stable distributions of exponent $\a$. 
The mixing distribution is given by the law of the limit for $n\to\infty$ of the random variables
\begin{align*}
  M_n^{(\a)}=\sum_{j=1}^n \b_{j,n}^\a ,
\end{align*}
which are defined in terms of the random weights defined in \eqref{recursion}.
The content of the following lemma is that $M_n^{(\a)}$ converges almost surely to a random 
variable  $M_\infty^{(\a)}$.
\begin{lemma}
  \label{Lemma2bis}
  Under condition \eqref{eq.defalpha},
  \begin{align}
    \label{eq.Misone}
    \E[M^{(\a)}_n] = \E[M^{(\a)}_{\n_t}]=1 \quad \mbox{for all $n\geq1$ and $t > 0$,}
  \end{align}
  and $M_n^{(\a)}$ converges almost surely to a non-negative random variable $M_\infty^{(\a)}$.

  In particular,
  \begin{itemize} 
  \item if $L^\a+R^\a=1$ a.s., 
    then $\QQ (s) \geq 0$ for every $s<\a$ and  $\QQ (s) \leq 0$ for every $s>\a$. 
    Moreover, $M_n^{(\a)}=M_\infty^{(\a)}=1$ almost surely;
  \item if $P\{L^\a+R^\a=1\}<1$ and if $\QQ (\gamma)<0$ for some $0<\gamma<\a$, 
    then $ M_\infty^{(\a)}=0$ almost surely; 
  \item if $P\{L^\a+R^\a=1\}<1$ and if $\QQ (\gamma)<0$ for some $\gamma>\a$, 
    then $M_\infty^{(\a)}$ is a non-degenerate random variable 
    with $\E[M_\infty^{(\a)}]=1$ and $\E[( M_\infty^{(\a)})^{\frac{\gamma}{\alpha}}]<+\infty$.
    Moreover, the characteristic function $\psi$ of  $M_\infty^{(\a)}$ is the unique solution
of
\begin{equation}
    \label{ligget3bis}
    \psi(\xi)=\E[\psi(\xi L^\a)\psi(\xi R^\a)]  \qquad (\xi \in \RE)
  \end{equation}
with $-i\psi'(0)=1$.
Finally, for any $p>\a$, 
    the moment $\E[( M_\infty^{(\a)})^{\frac{p}{\a}}]$ is finite if and only if $\QQ(p)<0$.
  \end{itemize}
\end{lemma}

We are eventually in the position to formulate our main results.
The first statement concerns the case where $\a\neq1$ and $\a\neq2$.
\begin{theorem}
  \label{thm2} { Assume that \eqref{eq.defalpha} holds 
   with $\a \in (0,1) \cup (1,2)$} and that $\QQ(\gamma)<0$ for 
 some $\gamma>0$.
  Moreover, let condition \eqref{stabledomain} be satisfied for $F=F_0$
  and let $X_0$ be centered if $\a>1$.
  Then $V_t$ converges in distribution, as $t\to +\infty$, to a random variable $V_\infty$ 
  with the following characteristic function
  \begin{equation}
    \label{characteristic}
    \phi_\infty(\xi) = \E[\exp(i \xi  V_\infty) ] 
    =\E[ \exp\{ -|\xi|^\a k M_{\infty}^{(\a)} (1-i \eta \tan(\pi\a/2)
    \operatorname{sign}\xi)   \} ]  
    \qquad (\xi \in \RE) ,
  \end{equation}
  where the parameters $k$ and $\eta$ are defined in \eqref{constant}.
  In particular, $V_\infty$ is a non-degenerate random variable if $c^++c^->0$ and $\gamma>\a$,
  whereas $V_\infty=0$ a.s. if $c^+=c^-=0$, or if $\gamma<\a$.
  Moreover, if $L^\a+R^\a=1$ a.s., then the distribution of $V_\infty$ is an $\a$-stable law.
  Finally, if $V_\infty$ is non-degenerate, 
  then $\E [|V_\infty|^{p}] <+\infty$ if and only if $p<\a$.
\end{theorem}
If $c^-=c^+$ then 
the limit distribution is a mixture of symmetric stable distributions. 
For instance this is true if $F_0$ is the distribution function of a symmetric random variable.

If $\a<1$ and $X_0\geq0$,
then clearly $c^-=0$ and the limit distribution is a mixture of positive stable distributions. 
Recall that a positive stable distribution
is characterized by its Laplace transform $s \mapsto \exp( - k s^\a)$;
hence, in this case,
\begin{align*}
  \E[\exp(- s  V_\infty) ]
  =\E[\exp\{ -s^\a  \bar k M_{\infty}^{(\a)}\}] 
  \quad\mbox{for all $s >0$, with $\bar k= c^+ \int_{0}^{+\infty}\frac{(1-e^{-y})}{y^{\a+1}}dy$.} 
\end{align*}

A  consequence of Theorem \ref{thm2} is that 
if $\E[ |X_0|^\a]<\infty$, 
then the limit $V_\infty$ is zero almost surely, since $c^+=c^-=0$.
The situation is different in the cases $\a=1$ and $\a=2$,
where $V_\infty$ is non-trivial provided that the first respectively  second moment of $X_0$ is finite.

\begin{theorem}
  \label{thm1} { Assume that \eqref{eq.defalpha} holds 
  with $\a=1$} and that $\QQ (\gamma)<0$
  for some $\gamma>0$. 
  If the initial condition  possesses a finite first moment $m_0=\E[X_0]$, 
  then $V_t$ converges in distribution, as $t\to +\infty$, to $V_\infty:=m_0 M_{\infty}^{(1)}$.
  In particular, $V_\infty$ is non-degenerate if $\gamma>1$ and $m_0\not=0$, whereas $V_\infty=0$ if $\gamma<1$.
  Moreover, if $L+R=1$ a.s., then $V_\infty=m_0$ a.s.
  Finally, if $V_\infty$ is non-degenerate and $p>1$, 
  then  $\E[|V_\infty|^{p}] <+\infty$ if and only if $\QQ(p)<0$.
\end{theorem}
We remark that under the hypotheses of the previous theorem, 
the first moment of the solution is preserved in time.
Indeed one has,
\begin{align*}
  \E[V_t] 
  = \E\big[\E\big[{\textstyle \sum_{j=1}^{\n_t} \b_{j,{\n_t}}X_j }\big| 
\nu_t,\,\b_{1,{\n_t}},\dots,\b_{\n_t ,{\n_t}} \big]\big] 
  = m_0  \E[M_{\n_t}^{(1)}] = m_0 ,
\end{align*}
where the last equality follows from \eqref{eq.Misone}.

Theorem \ref{thm1} above is the most natural generalization 
of the results in \cite{MatthesToscani},
where the additional condition $\E[ |X_0|^{1+\epsilon} ]<\infty$ for some $\epsilon>0$ has been assumed.
The respective statement for $\a=2$ reads as follows.

\begin{theorem}
  \label{thm3-bis} { Assume that \eqref{eq.defalpha} holds 
  with $\a = 2$ } and that $\QQ (\gamma)<0$ for some $\gamma>0$. 
  If  $\E[X_0]=0$ and $\s^2=\E[X_0^2]<+\infty$,
  then $V_t$ converges in distribution, as $t\to +\infty$, to a random variable $V_\infty$ 
  with characteristic function 
  \begin{align}
    \label{characteristic-gauss}
    \phi_\infty (\xi) = \E[\exp( i \xi V_\infty) ]
    =\E\left[\exp( -  \xi^2 \frac{\s^2}{2}  M_\infty^{(2)} )\right]
    \qquad (\xi \in \RE).
  \end{align}
  In particular, $V_\infty$ is a non-degenerate random variable if $\gamma>2$ and $\s^2>0$,
  whereas $V_\infty=0$ a.s. if $\gamma<2$.
  Moreover, if $L^2+R^2=1$ a.s., then  $V_\infty$ is a Gaussian random variable.
  Finally, if $V_\infty$ is non-degenerate and $p>2$, 
  then $\E[ |V_\infty|^{p} ] <+\infty$ if and only if $\QQ(p)<0$.
\end{theorem}
Some additional properties of the solution $V_t$ should be mentioned:
centering is obviously propagated from the initial condition $X_0$
to the solution $V_t$ at all later times $t\geq0$.
Moreover, under the hypotheses of the theorem, 
the second moment of the solution is preserved in time.
Indeed, taking into account that the $X_i$ are independent and centered,
\begin{align*}
  \E[V_t^2] 
  = \E\big[\E\big[{\textstyle \sum_{j,k=1}^{\n_t} \b_{j,{\n_t}}\b_{k,{\n_t}}X_jX_k }\big| \nu_t,
\,\b_{1,{\n_t}},\dots,\b_{\n_t ,{\n_t}}   \big]\big] 
  = \s^2 \E[M_{\n_t}^{(2)}] = \s^2,
\end{align*}
where we have used \eqref{eq.Misone} in the last step.

The technically most difficult result concerns the situation $\a=1$
for an initial condition of infinite first moment.
Weak convergence to a limit can still be proven if $\E[ |X_0| ]=\infty$,
but the law of $X_0$ belongs to the domain of normal 
attraction of a 1-stable distribution.
However, a suitable time-dependent centering 
needs to be applied to the random variables $V_t$.

\begin{theorem}
  \label{thm2-bis} 
  { Assume that \eqref{eq.defalpha} holds with $\a=1$} and that $\QQ (\gamma)<0$ for some $\gamma>0$.
  Moreover, let  the condition \eqref{stabledomain} be satisfied for $F=F_0$.
  Then the random variable
\begin{equation}\label{centr-case1}
      V_t^* := V_t - \sum_{j=1}^{\nu_t} q_{j,\nu_t} , 
    \quad\mbox{where}\quad q_{j,n}:=\int_{\RE} \sin(\b_{j,n}x)dF_0(x),
  \end{equation}
  converges in distribution to a limit $V^*_\infty$ 
  with characteristic function
  \begin{equation}
    \label{characteristic-bis}
    \phi_\infty(\xi) = \E[\exp(i \xi  V^*_\infty) ]
    =\E[ \exp\{ -|\xi| k M_{\infty}^{(1)} (1+2i\eta/\pi \log|\xi| \operatorname{sign}\xi )  \} ]  
    \qquad (\xi \in \RE)
  \end{equation}
  where the parameters $k$ and $\eta$ are defined in \eqref{constant}.
  In particular, $V_\infty$ is a non-degenerate random variable if $c^++c^->0$ and $\gamma>1$,
  whereas $V_\infty=0$ a.s. if $c^+=c^-=0$, or if $\gamma<1$.
  Moreover, if $L+R=1$ a.s., then the distribution of $V_\infty$ is a $1$-stable law.
  Finally, if $V_\infty$ is non-degenerate, 
  then $\E [|V_\infty|^{p}] <+\infty$ if and only if $p<1$.
\end{theorem}

\subsection{Rates of convergence in Wasserstein metrics}
\label{sct.wasserstein}

Recall that the Wasserstein distance of order $\gamma>0$ 
between two random variables $X$ and $Y$ is defined by
\begin{align}
  \label{eq.wasserstein}
  \W_\gamma(X,Y):=\inf_{(X',Y')}(\E|X'-Y'|^\gamma)^{1/\max(\gamma,1)} . 
\end{align}
The infimum is taken over all pairs $(X',Y')$ of real random variables
whose marginal distribution functions are the same as those of $X$ and $Y$, respectively.
In general, the infimum in \eqref{eq.wasserstein} may be infinite;
a sufficient (but not necessary) condition for finite distance is
that both $\E[|X|^\gamma]<\infty$ and $\E[|Y|^\gamma]<\infty$. For more
information on Wasserstein distances see, for example, \cite{Rachev1991}. 

Recall that the $V_t$ are random variables whose characteristic functions $\phi(t)$ solve
the initial value problem \eqref{eq.boltzivp} for the Boltzmann equation for $t\geq0$,
and $V_\infty$ is the limit in distribution  of $V_t$ as $t\to\infty$.
\begin{proposition}
  \label{PropW-2}
  Assume \eqref{eq.defalpha} and $\QQ (\gamma)<0$, for some $\gamma$ with $1\leq \a<\gamma\leq2$
  or $\a<\gamma \leq 1$.
  Assume further that \eqref{stabledomain} holds if $\a\not=1$,
  or that $\E[|X_0|^{\gamma}]<+\infty$ if $\a=1$, respectively.
  Then 
  \begin{equation}
    \label{main_bound}
    \W_{\gamma}(V_t,V_\infty) \leq A \W_{\gamma}(X_0,V_\infty) e^{-Bt|\QQ (\gamma)|} ,
  \end{equation}
  with $A=B=1$ if $\gamma\leq1$, or $A=2^{1/\gamma}$ and $B=1/\gamma$ otherwise.
\end{proposition}

Clearly, the content of Proposition \ref{PropW-2} is void unless 
\begin{align}
  \label{eq.wfinite}
  \W_{\gamma}(X_0,V_\infty) < \infty .
\end{align}
In the case $\a=1$, the hypothesis $\E|X_0|^{\gamma}<+\infty$ guarantees \eqref{eq.wfinite}. 
In all other cases, \eqref{eq.wfinite} is a non-trivial requirement since, 
by Theorem \ref{thm2}, either $V_\infty=0$ or $\E[|V_\infty|^\alpha]=+ \infty$.
The following Lemma provides a sufficient criterion for \eqref{eq.wfinite},
tailored to the situation at hand.
\begin{lemma}
  \label{lemma-tail2}
  Assume, in addition to the hypotheses of Proposition \ref{PropW-2}, 
  that $\gamma<2\a$ 
  and that $F_0$ 
  satisfies hypothesis \eqref{stabledomain} in the more restrictive sense that
  there exists a constant $K>0$ and some $0<\epsilon<1$ with
  \begin{align}
    \label{eq.Ftail}
    |1-c^+x^{-\alpha} - F_0(x)| & < K x^{-(\alpha+\epsilon)} \quad\mbox{for $x>0$}, \\
    \label{eq.Ftail2}
    |F_0(x) - c^-(-x)^{-\alpha}| & < K (-x)^{-(\alpha+\epsilon)} \quad\mbox{for $x<0$}.
  \end{align}
  Provided that $\gamma<\a/(1-\eps)$ 
  it follows $\W_{\gamma}(X_0,V_\infty)<\infty$, 
  and then estimate \eqref{main_bound} is non-trivial.
\end{lemma}

\subsection{Strong convergence of densities}
\label{sct.strong}
As already mentioned in the introduction, 
under suitable hypotheses, 
the probability densities of $\mu(t)$ exist and converge strongly 
in the Lebesgue spaces $L^1(\RE)$ and $L^2(\RE)$.

\begin{theorem}
  \label{thm.strongnew}
  For given $\a\in(0,1)\cup(1,2]$,
  let the hypotheses of Theorem \ref{thm2} or Theorem \ref{thm3-bis}  
 hold with $\gamma >\a$.
  Assume further that \eqref{stabledomain} holds with $c^-+c^+>0$ if $\a<2$, 
  so that the $V_t$ converges in distribution, as $t \to +\infty$,
   to a non-degenerate 
  limit $V_\infty$.
  Moreover assume also  that
  \begin{enumerate}
  \item[(H1)] $L^r+R^r\geq1$ a.s. for some $r>0$,
  \item[(H2)] $X_0$ possesses a density $f_0$ with finite Linnik-Fisher functional, 
    i.e. $h:=\sqrt{f_0}\in H^1(\RE)$, or equivalently, its Fourier transform $\fh$ satisfies
    \[ \int_\RE |\xi|^2 \big|\fh(\xi)\big|^2 \,d\xi < +\infty. \]
  \end{enumerate}
  Then, the random variable $V_t$ possesses a density $f(t)$ for all $t\geq0$,
  $V_\infty$ has a density $f_\infty$, and
  the $f(t)$ converges, as $t \to +\infty$, to $f_\infty$ in any $L^p(\RE)$ 
   with $1\leq p\leq2$, that is
   \[
   \lim_{t \to +\infty}\int_\RE|f(t;v)-f_\infty(v)|^pdv=0.
   \]
\end{theorem}

\begin{remark}\label{rem3}
  Some comments on the hypotheses (H1) and (H2) are in order.
  \begin{itemize}
  \item In view of $\QQ(\a)=1$, condition (H1) can be satisfied only if $r<\a$.
    Notice that (H1) becomes the weaker the smaller $r>0$ is;
    in fact, the sets $\{(x,y)| x^r+y^r\geq1\}\subset\RE^2$ exhaust the first quadrant as $r\searrow0$.
  \item The smoothness condition (H2) is not quite as restrictive as it may seem. 
    For instance, recall that the convolution of any probability density $f_0$ 
    with an arbitrary ``mollifier'' of finite Linnik-Fisher functional (e.g. a Gaussian)
    produces again a probability density of finite Linnik-Fisher functional.
  \end{itemize}
\end{remark}

\section{Proofs }\label{sct.proofs}

We continue to assume that the law of the random vector $(L,R)$ is given
and satisfies \eqref{eq.defalpha} with $\a\in(0,2]$, implying $\QQ(\a)=0$.

\subsection{Properties of the weights $\b_{j,n}$ (Lemma \ref{Lemma2bis})}
In this subsection we shall prove a generalization of a useful result
obtained in  \cite{GabettaRegazzini2006a}. 
Set
\begin{equation}\label{G_n}
  G_n=(I_1,\dots,I_{n-1},
  L_1,R_1,\dots,L_{n-1},R_{n-1}).
\end{equation}
and denote by $\CG_n$  the $\s$-algebra generated by $G_n$.

\begin{proposition}\label{PropConst} 
  If $\E[L^s+R^s] <+\infty$ for some $s>0$, then
  \[
  \E[M_n^{(s)}]=\E[\sum_{j=1}^n \beta_{j,n}^s]=\frac{\Gamma(n+\QQ (s))}{\Gamma(n)\Gamma(\QQ (s)+1)}
  \]
  and
  \[
  \E[M_{\n_t}^{(s)}]=\sum_{n\geq1}e^{-t}(1-e^{-t})^{n-1}\E[M_n^{(s)}]=e^{t\QQ (s)}.
  \]
  If in addition $\QQ (s)=0$, for some $s>0$, then 
  $M_n^{(s)}$ is a martingale with respect to $(\CG_n)_{n\geq 1}$. 
\end{proposition}

\begin{proof} 
  Recall that $(\b_{1,1},\b_{1,2},\b_{2,2},\dots,\b_{n,n})$ is $\CG_n$--measurable, see 
 (\ref{recursion}).
  We first prove that $\E[M_{n+1}^{(s)}|\CG_n]=M_n^{(s)}(1+\QQ (s)/n)$, 
  which implies
  that $M_n^{(s)}$ is a $(\CG_n)_n$--martingale whenever $\QQ (s)=0$,
  since $M_n^{(s)} \geq 0$ and, as we will see, $\E[M_n^{(s)}]<+\infty$ for every $n\geq 1$.
  To prove the claim write
  \begin{align*}
    \E[M_{n+1}^{(s)}|\CG_n]
    & =\E\Big[\sum_{i=1}^{n}\J\{I_n=i\}\sum_{j=1}^{n+1}\beta_{j,n+1}^s\Big|\CG_n\Big] \\
    &=\E\Big[\sum_{i=1}^{n}\J \{I_n=i\}  \Big(\sum_{j=1,\dots,n+1,j\not=i,i+1}
    \beta_{j,n+1}^s+\beta_{i,n+1}^s +\beta_{i+1,n+1}^s\Big)\Big|\CG_n\Big]\\
    &=\E\Big[\sum_{i=1}^{n}\J \{I_n=i\}  \Big(\sum_{j=1}^n \beta_{j,n}^s
    +\beta_{i,n}^s(L_n^s+R_n^s-1)\Big)\Big|\CG_n\Big]\\
    &=M_n^{(s)}+\QQ (s)\E\Big[\sum_{i=1}^{n}\J \{I_n=i\}\beta_{i,n}^s|\CG_n\Big] \\
    &= M_n^{(s)} + \QQ(s) \sum_{i=1}^{n}\beta_{j,n}^s \E[\J \{I_n=i\}]
    =M_n^{(s)}(1+\QQ (s)/n).
  \end{align*}
  Taking the expectation of both sides one gets
  \[
  \E[M_{n+1}^{(s)}]=  \E[M_{n}^{(s)}](1+\frac{1}{n}\QQ (s)).
  \]
  Since $\E[M_2^{(s)}]=\QQ (s)+1$ it follows easily that
  \[
  \E[M_{n}^{(s)}]=\prod_{i=1}^{n-1}(1+\QQ (s)/i)
  =\frac{\Gamma(n+\QQ (s))}{\Gamma(n)\Gamma(\QQ (s)+1)}.
  \]
  To conclude the proof  use formula 5.2.13.30 in \cite{prudnikov}. 
\end{proof}


\begin{lemma}\label{Lemma1} 
  If $\QQ (\gamma)<0$ for some $\gamma>0$,
  then
  \[
  \b_{(n)}:=\max_{1 \leq j \leq n} \b_{j,n}
  \]
  converges to zero in probability as $n \to +\infty$. 
\end{lemma}

\begin{proof} 
  Observe that, for every $\eps>0$,
  \[
  P\{ \b_{(n)} >\eps \} \leq P \Big \{ \sum_{j=1}^n \b_{j,n}^\gamma \geq \eps^\gamma\Big  \}
  \]
  and hence, by Markov's inequality and Proposition \ref{PropConst},
  \[
  P\{ \b_{(n)} >\eps \} \leq  \frac{1}{\eps^\gamma} \E [M_n^{(\gamma)} ]
  =\frac{1}{\eps^\gamma}\frac{\Gamma(n+\QQ (\gamma))}{\Gamma(n)\Gamma(\QQ (\gamma)+1)} \leq C \frac{1}{\eps^\gamma} n^{\QQ (\gamma)}.
  \]
  The last expression tends to zero as $n\to\infty$ because $\QQ(\gamma)<0$.
\end{proof}

\begin{proof}[Proof of Lemma \ref{Lemma2bis}]
  Since $\QQ (\a)=0$, 
  the random variables $M_n^{(\a)}$ form a positive martingale with respect to $(\CG_n)_n$ by
  Proposition \ref{PropConst}.
  By the martingale convergence theorem, see e.g. Theorem 19 in 
  Chapter 24 of \cite{fristedgray},
  it converges a.s. to a positive random variable $M_\infty^{(\a)}$ with $\E[M_\infty^{(\a)}] \leq \E[M_1^{(\a)}]=1$.
  The goal of the following is to determine the law of $M_\infty^{(\alpha)}$
  in the different cases we consider.

  First, suppose that $L^\a+R^\a=1$ a.s. 
  It follows that $L^\a\leq 1$ and $R^\a \leq 1$ a.s., 
  and hence $\QQ (s) \leq \QQ(\a) = 0$ for all $s>\a$.
  Moreover, it is plain to check that $M_n^{(\a)}=1$ a.s. for every $n$, 
  and hence $M_\infty^{(\a)}=1$ a.s. 

  Next, assume that $\QQ(\gamma)< 0 $ for $\gamma<\a$. 
  Minkowski's inequality and Proposition \ref{PropConst} give
  \[
  \E \big[ (M_n^{(\a)})^{\gamma/\a}\big] \leq \E [\sum_{j=1}^n \b_{j,n}^\gamma ] = 
  \frac{\Gamma(n+\QQ (\gamma))}{\Gamma(n)\Gamma(\QQ (\gamma)+1)} \leq C n^{\QQ (\gamma)}.
  \]
  Hence, $M_n^{(\a)}$ converges a.s. to $0$. 
  
  It remains to treat the case with $\QQ(\gamma)<0$ and $\gamma>\a$. 
  Since $\QQ (\cdot)$ is a convex function satisfying $\QQ (\a)=0$ and $\QQ (\gamma)<0$ with $\gamma>\a$,
  it is clear that $\QQ'(\a)<0$; 
  also, we can assume without loss of generality that $\gamma<2\a$.
  Further, by hypothesis,
  \[
  \E[(L^\a+R^\a)^{1+(\gamma/\a-1)}]\leq 2^{\gamma/\a-1}\E[L^\gamma+R^\gamma]<+\infty .
  \]
  Hence, one can resort to  Theorem 2(a) of \cite{DurrettLiggett1983}
  --- see also  Corollaries 1.1, 1.4  and 1.5 in \cite{Liu1998} ---
  which provides existence and uniqueness of a probability distribution 
  $\nu_\infty\not=\delta_0$ on $\RE^+$, 
  whose characteristic function $\psi$ 
  is a  solution of  equation (\ref{ligget3bis}),
  with $\int_{\RE^+} x \nu_\infty(dx)=1$. 
  Moreover, Theorem 2.1 in \cite{Liu2000} ensures that 
  $\int_{\RE^+}x^{\gamma/\a} \nu_\infty(dx)<+\infty$ 
  and, more generally, that $\int_{\RE^+}x^{p/\a} \nu_\infty(dx)<+\infty$ for some $p>\a$ 
  if and only if $\QQ(p)<0$. 
 
  Consequently, our goal is to prove that the law of $M_\infty^{(\a)}$ is $\nu_\infty$.
  In what follows, enlarge the space $(\Omega,\CF,P)$ in order to contain
  all the random elements needed. 
  In particular, let $(M_j)_{j \geq 1}$ be a sequence of independent random variables 
  with common characteristic function $\psi$,
  such that $(M_j)_{j \geq 1}$ and $(G_n)_{n \geq 1}$, defined in (\ref{G_n}), are independent. 
  Recalling that $\psi$ is a solution of  (\ref{ligget3bis}) 
  it follows that, for every $n\geq 2$,
  \begin{align*}
    & \E\Big[\exp\Big\{i\xi \sum_{j=1}^n \b_{j,n}^\a M_j \Big\}\Big]  \\
    & = \sum_{k=1}^{n-1} \frac{1}{n-1}
    \E \Big [ \exp  \Big \{i\xi \Big ( \sum_{j=1}^{k-1} \b_{j,n-1}^\a M_j
    +\b_{k,n-1}^\a \underbrace{(L_{n-1}^\a M_k + R_{n-1}^\a M_{k+1}  )}_{=^d\,M_k}
    + \sum_{j=k+1}^{n} \b_{j,n-1}^\a M_j \Big)  \Big\} \Big]
    \\
    &= \E\Big[\exp\Big\{i\xi \sum_{j=1}^{n-1} \b_{j,n-1}^\a M_j \Big\}\Big]. 
  \end{align*}
  By induction on $n\geq2$, this shows that $\sum_{j=1}^n M_j \b_{j,n}^\a$ has the same law as $M_1$, 
  which is $\nu_\infty$.
  Hence
  \begin{align*}
    \W_{\gamma/\a}^{\gamma/\a}  (M_n^{(\a)},M_1) 
    \leq \E \Big[ \Big |\sum_{j=1}^{n} \b_{j,n}^\a - 
  \sum_{j=1}^{n} M_j \b_{j,n}^\a \Big|^{\gamma/\a} \Big] 
    =\E \Big[
    \E \bigg[ \Big |\sum_{j=1}^{n} (1 - M_j) \b_{j,n}^\a \Big|^{\gamma/\a}
    \bigg|\CG_n \bigg] \Big]. 
  \end{align*}
  We shall now employ the following result from \cite{BahrEsseen1965}.
  {\em Let $1<\eta \leq 2$, and assume that $Z_1, \dots , Z_n$ are 
    independent, centered random variables  and $\E|Z_j|^\eta<+\infty$.
    Then}
  \begin{equation}
    \label{bahr-essen2}
    \E\Big |\sum_{j=1}^n Z_j \Big |^\eta \leq  2  \sum_{j=1}^n \E| Z_j |^\eta .
  \end{equation}
  We apply this result with $\eta=\gamma/\a$ and $Z_j=\b_{j,n}^\a(1-M_j)$, 
  showing that
  \begin{align*}
    \E\Big [\Big |\sum_{j=1}^{n} (1 -M_j) \b_{j,n}^\a \Big |^{\gamma/\a} \Big |\CG_n \Big ]
    &\leq 2 \sum_{j=1}^{n} \b_{j,n}^{\gamma} \E\Big [\Big | 1 - M_j \Big |^{\gamma/\a}  
    \Big |\CG_n \Big ] = 2  \sum_{j=1}^{n} \b_{j,n}^{\gamma} \E|1-M_1|^{\gamma/\a}
  \end{align*}
  almost surely.
  In consequence,
  \begin{equation}
    \label{eqstima}
    \W_{\gamma/\a}^{\gamma/\a}  (M_n^{(\a)},M_1) \leq  
    2  \E \Big[\sum_{j=1}^{n} \b_{j,n}^{\gamma}\Big ] \E\Big [|1-M_1|^{\gamma/\a}\Big ].
  \end{equation}
  By means of Proposition \ref{PropConst}, one obtains
  \[
  \W_{\gamma/\a}^{\gamma/\a}  (M_n^{(\a)},M_1) \leq  C' n^{\QQ (\gamma)}.
  \]
  This proves that the law of $M_n^{(\a)}$ converges with respect to the $\W_{\gamma/\a}$ metric 
  -- and then also weakly --
  to the law of $M_1$.
  Hence, $M_\infty^{(\a)}$ has law $\nu_\infty$. 
  The fact that $M_\infty^{(\a)}$ is non-degenerate, provided $L^\a+R^\a=1$ does {\em not} hold a.s.,
  follows immediately.
\end{proof}

\subsection{Proof of convergence for $\alpha\neq1$ (Theorems \ref{thm2} and \ref{thm3-bis})}

Denote by $\CB$ the $\s$-algebra generated by $\{\beta_{j,n}: n \geq 1, j=1,\dots , n \}$.
The proof of Theorems \ref{thm2} and \ref{thm3-bis} is 
essentially an application of  the central limit theorem 
to the conditional law of 
\begin{align*}
  W_n := \sum_{j=1}^{n}\beta_{j,n}X_j
\end{align*}
given $\CB$. 
Set
\begin{align*}
  Q_{j,n}(x):=F_0\big(\b_{j,n}^{-1}x\big),
\end{align*}
where, by convention, $F_0(\cdot/0):=\J_{[0,+\infty)}(\cdot)$. 
In this subsection we will use the functions:
  \begin{align*}
    \z_n(x) &:= 
    \J\{x<0\} \sum_{j=1}^{n} Q_{j,n}(x) + \J\{x>0\} \sum_{j=1}^{n} (1-Q_{j,n}(x)) \qquad (x \in \RE) \\
    \s_{n}^2(\eps) &:=  
    \sum_{j=1}^{n}\Big\{ \int_{(-\eps,+\eps]} x^2\,dQ_{j,n}(x)-\Big( \int_{(-\eps,+\eps]} x\,dQ_{j,n}(x)\Big)^2 \Big\} \qquad (\eps>0) \\
    \eta_{n} &: =   \sum_{j=1}^{n}\Big\{ 1- Q_{j,n}(1) -Q_{j,n}(-1) + \int_{(-1,1]} x\, dQ_{j,n}(x)\Big \}. \\
  \end{align*}
In terms of $Q_{j,n}$,
the conditional distribution function $F_n$ of $W_n$ given $\CB$
is the convolution,
\[
F_n = Q_{1,n} *\cdots *Q_{n,n}.
\]
To start with, we show that the $Q_{j,n}$s satisfy 
the uniform asymptotic negligibility (UAN) assumption \eqref{UAN} below.
\begin{lemma}
  \label{lemma3} 
  Let the assumptions of Theorem \ref{thm3-bis} or Theorem \ref{thm2-bis} be in force. 
  Then, for every divergent sequence $(n')$ of integer numbers, 
  there exists a divergent subsequence $(n'') \subset (n')$ 
  and a set $\Omega_0$ of probability one such that 
  \begin{equation}\label{conMn}
    \begin{split}
      &\lim_{n'' \to +\infty }M_{n''}^{(\a)}(\omega)=M_\infty^{(\a)}(\omega) < \infty, \\
      &\lim_{n'' \to +\infty }\b_{(n'')}(\omega)=0,\\
    \end{split}
  \end{equation}
  holds for every $\omega\in\Omega_0$.
  Moreover, for every $\omega\in\Omega_0$ and for every $\eps>0$
  \begin{equation}\label{UAN}
    \lim_{n'' \to +\infty} 
    \max_{1\leq j \leq n''}\Big\{ 1-Q_{j,n''}(\eps) + Q_{j,n''}(-\eps)  \Big\}=0.
  \end{equation}
\end{lemma}
\begin{proof}
  The existence of a sub-sequence $(n'')$ and a set $\Omega_0$
  satisfying (\ref{conMn}) is a direct consequence 
  of Lemmata \ref{Lemma1} and  \ref{Lemma2bis}.
  To prove (\ref{UAN}) note that, for $0< \a \leq 2$ and $\a \not = 1$,
  \[
  \max_{1\leq j \leq n''} \Big(1-Q_{j,n''}(\eps)+Q_{j,n''}(-\eps)\Big )
  \leq 1-F_0\Big (\frac{\eps}{\beta_{(n'')}}\Big)
  +F_0\Big (\frac{-\eps}{\beta_{(n'')}}\Big).
  \]
  The claim hence follows from (\ref{conMn}).
\end{proof}
\begin{lemma}
  \label{lemma4_2} 
  Let the assumptions of Theorem \ref{thm2} be in force.
  Then for every divergent sequence $(n')$ of integer numbers 
  there exists a divergent subsequence $(n'') \subset (n')$
  and a measurable set $\Omega_0$ of probability one 
  such that 
  \begin{align}
    \label{eq.char02}
    \lim_{n'' \to +\infty} \E[ e^{ i \xi W_{n''}}|\CB](\omega)    
    =\exp\{ -|\xi|^\a k M_{\infty}^{(\a)}(\omega) (1-i \eta \tan(\pi\a/2)\operatorname{sign}\xi)   \}   
    \quad (\xi \in \RE )
  \end{align}
  for every $\omega$ in $\Omega_0$.
\end{lemma}
\begin{proof}
  Let $(n'')$ and $\Omega_0$ be the same as in Lemma \ref{lemma3}.
  To prove \eqref{eq.char02}, 
  we apply the central limit theorem for every $\omega\in\Omega_0$
  to the conditional law of $W_{n''}$ given $\CB$. 

  For every $\omega$ in $\Omega_0$,
  we know that $F_{n''}$ is a convolution of probability distribution functions 
  satisfying the asymptotic negligibility assumption (\ref{UAN}).
  Here, we shall use the general version of the central limit theorem as presented
  e.g., in Theorem 30 in Section 16.9 and in Proposition 11 in Section 17.3 of \cite{fristedgray}.
  According to these results, the claim \eqref{eq.char02} follows if, 
  for every $\omega\in\Omega_0$, 
  \begin{align}
    \label{condition1-uno}
    \lim_{n'' \to +\infty } \z_{n''}(x) &= \frac{c^{+}M_\infty^{(\a)}}{x^\a}  \qquad (x>0), \\
    \label{condition1-due}
    \lim_{n'' \to +\infty} \z_{n''}(x) &= \frac{c^{-} M_\infty^{(\a)}}{|x|^\a} \qquad (x<0), \\
    \label{clt3condBis}
    \lim_{\eps \to 0^+}\limsup_{n'' \to +\infty} \s^{2}_{n''}(\eps) &= 0 , \\
    \label{condition3}
    \lim_{n'' \to +\infty} \eta_{n''} &= \frac{1}{1-\a} M_\infty^{(\a)}(c^+-c^-)
  \end{align}
  are simultaneously satisfied.

  In what follows we assume that $P\{L=0\}=P\{R=0\}=0$,
  which yields that $\b_{j,n}>0$ almost surely. 
  The general case can be treated with minor modifications. 

  In order to prove (\ref{condition1-uno}),
  fix some $x>0$, and observe that
  \[
  \z_{n''}(x)=\sum_{j=1}^{n''} [1-F_0(\b_{j,n''}^{-1}x)]
  = \sum_{j=1}^{n''} [1-F_0(\b_{j,n''}^{-1}x)] (\b_{j,n''}^{-1}x)^\a \frac{\b_{j,n''}^\a}{x^\a}.
  \]
  Since $\lim_{y \to +\infty}(1-F_0(y))y^\a=c^+$ by assumption \eqref{stabledomain}, 
  for every $\eps>0$ there exists a $Y=Y(\eps)$ such that if $y>Y$, 
  then $c^+-\eps \leq (1-F_0(y))y^\a \leq c^++\eps $. 
  Hence if $x >\b_{(n'')}Y$, then
  \begin{align}
    \label{eq.thirtyeightandonehalf}
    x^{-\alpha}(c^+-\eps) M_{n''}^{(\a)} \leq \sum_{j=1}^{n''}(1-F_0(\b_{j,n''}^{-1}x))\leq x^{-\alpha}(c^++\eps) M_{n''}^{(\a)}.
  \end{align}
  In view of \eqref{conMn}, the claim (\ref{condition1-uno}) follows immediately.
  Relation (\ref{condition1-due}) is proved in a completely analogous way. 

  In order to prove (\ref{clt3condBis}), 
  it is clearly sufficient to show that for every $\eps>0$
  \begin{align}
    \label{eq.sigma}
    \limsup_{n'' \to +\infty}
   \sum_{j=1}^{n''} \int_{(-\epsilon,+\epsilon]} x^2 dQ_{j,n''}(x) \leq C M_{\infty}^{(\a)}\epsilon^{2-\alpha} 
  \end{align}
  with some constant $C$ independent of $\epsilon$.
  Recalling the definition of $Q_{j,n}$,
  an integration by parts gives
  \begin{align*}
    \int_{(0,\epsilon]} x^2 dF_0\big( \b_{j,n}^{-1} x \big) 
    = - \epsilon^2\big[1-F_0\big( \b_{j,n}^{-1}\epsilon \big)\big] 
    + 2 \int_0^\epsilon x \big[1-F_0\big( \b_{j,n}^{-1} x \big)\big]\,dx ,
  \end{align*}
  and similarly for the integral from $-\epsilon$ to zero.
  With
  \begin{align}
    \label{eq.defK}
    K := \sup_{x>0} x^\alpha [1-F(x)] + \sup_{x<0} (-x)^\alpha F(x),
  \end{align}
  which is finite by hypothesis \eqref{stabledomain},
  it follows that
  \begin{align*}
    \int_{(-\epsilon,+\epsilon]} x^2 dF_0\big( \b_{j,n}^{-1} x \big) 
    \leq 2K\epsilon^{2}(\b_{j,n}^{-1}\epsilon)^{-\a} + 4K\b_{j,n}^\a \int_0^\epsilon x^{1-\a}\,dx
    \leq 2K\Big( 1 + \frac{2}{2-\a} \Big) \b_{j,n}^\a \epsilon^{2-\a} .
  \end{align*}
  To conclude \eqref{eq.sigma},
  it suffices to recall that $\sum_j \b_{j,n''}^\a = M_{n''}^{(\a)} \to M_\infty^{(\a)}$ by \eqref{conMn}.

  In order to prove (\ref{condition3}), 
  we need to distinguish if $0<\a<1$, or if $1<\a<2$.
  In the former case,
  integration by parts in the definition of $\eta_{n''}$ reveals
  \[
  \eta_{n''} = \int_{-1}^1 \z_{n''}(x) dx.
  \]
  Having already shown \eqref{condition1-uno} and \eqref{condition1-due},
  we know that the integrand converges pointwise with respect to $x$.
  The dominated convergence theorem applies since,
  by hypothesis \eqref{stabledomain},
  \begin{align}
    \label{eq.z}
    |\z_{n''}(x) |\leq K|x|^{-\a} \sup_{n''}M_{n''}^{(\a)}
  \end{align}
  with the constant $K$ defined in \eqref{eq.defK};
  observe that $|x|^{-\a}$ is integrable on $(-1,1]$ since we have assumed $0<\a<1$.
  Consequently,
  \begin{align*}
    \lim_{n''\to\infty} \eta_{n''}
    = c^{-}M_\infty^{(\a)} \int_{-1}^0 |x|^{-\a}\,dx + c^+ M_\infty^{(\a)} \int_0^1 |x|^{-\a}\,dx 
    = \frac{c^+ - c^-}{1-\alpha}M_\infty^{(\a)} .
  \end{align*}
  It remains to check (\ref{condition3}) for $1<\a<2$. 
  Since  $\int_\RE x\,dQ_{j,n''}(x)=0$, one can write
  \[
  \eta_{n''}=\eta_{n''}-\sum_{j=1}^{n''}\int_\RE x\,dQ_{j,n''}(x)
  = -\sum_{j=1}^{n''}\int_{(-\infty,-1]}(1+x)dQ_{j,n''}(x)
  -\sum_{j=1}^{n''}\int_{(1,+\infty)}(x-1)dQ_{j,n''}(x).
  \]
  Similar as for $0<\a<1$, integration by parts reveals that
  \begin{align}
    \label{eq.etafromzeta}
    \eta_{n''} = \int_{\{|x|>1\}} \z_{n''}(x)\,dx.
  \end{align}
  From this point on, the argument is the same as in the previous case:
  \eqref{condition1-uno} and \eqref{condition1-due} provide pointwise convergence of the integrand;
  hypothesis \eqref{stabledomain} leads to \eqref{eq.z}, 
  which guarantees that the dominated convergence theorem applies,
  since $|x|^{-\a}$ is integrable on the set $\{ |x| > 1 \}$.
  It is straightforward to verify that the integral of the pointwise limit
  indeed yields the right-hand  side of \eqref{condition3}.
\end{proof}

\begin{proof}[Proof of Theorem \ref{thm2}] 
  By Lemma \ref{lemma4_2} and the dominated convergence theorem,
  every divergent sequence $(n')$ of integer numbers 
  contains a divergent subsequence $(n'') \subset (n')$ 
  for which
  \begin{align}
    \label{eq.thm2limit}
    \lim_{n'' \to +\infty} \E[ \exp\{ i\xi W_{n''} \}]
    =\E\big[\exp\big\{-|\xi|^\a k  M_{\infty}^{(\a)}(1-i \eta \tan(\pi\a/2) \operatorname{sign}\xi)\big\}\big],
  \end{align}
  where the limit is pointwise in $\xi \in \RE$. 
  Since the limiting function is independent of the arbitrarily chosen sequence $(n')$,
  a classical argument shows that \eqref{eq.thm2limit} is true with $n\to+\infty$ in place of 
  $n''\to+\infty$.
  In view of Proposition \ref{Prop:probint}, the stated convergence follows. 
  
  By  Lemma \ref{Lemma2bis},
  the assertion about (non)-degeneracy of $V_\infty$ follows immediately 
  from the representation \eqref{characteristic}.
  To verify the claim about moments for $\gamma>\a$,
  observe that \eqref{characteristic} implies that
  \begin{equation}
    \label{eq.momcalc}
      \E[|V_\infty|^p] = \int_\RE |x|^p dF_\infty (x)
     = \E\big[\big(M^{(\a)}_\infty\big)^p\big] \int_\RE |u|^p dG_\a(u),
     \end{equation}
where $G_\a$ is the distribution function of the centered $\a$-stable law
with characteristic function $\hat g_\a$ defined in \eqref{chaSta}.

  The $p$-th moment of $M^{(\a)}_\infty $ is finite at least for all $p<\gamma$ by Lemma \ref{Lemma2bis}.
  On the other hand, the $p$-th moment of $G_\a$ is finite if and only if $p<\a$.
\end{proof}

The following lemma replaces Lemma \ref{lemma4_2} in the case $\a=2$.

\begin{lemma}\label{lemma4_5}
  Let the assumptions of Theorem \ref{thm3-bis} hold.
  Then for every divergent sequence $(n')$ of integer numbers,
  there exists a divergent subsequence $(n'') \subset (n')$ 
  and a set $\Omega_0$ of probability one
  such that 
  \[
  \lim_{n'' \to +\infty} \E[ e^{ i \xi W_{n''}}|\CB](\omega)   
  =e^{- \xi^2 \frac{\s^2}{2} M_{\infty}^{(2)} (\omega)} \quad (\xi \in \RE)
  \]
  for every $\omega$ in $\Omega_0$.
\end{lemma}

\begin{proof} 
  Let $(n'')$ and $\Omega_0$ have the properties stated in Lemma \ref{lemma3}.
  The claim follows if for every $\omega$ in $\Omega_0$, 
  \begin{align}
    \label{condition1-uno-bis}
    \lim_{n'' \to +\infty} \z_{n''}(x) 
    &= 0 \qquad (x\not =0), \\
    \label{clt3condtris}
    \lim_{\eps \to 0^+}\lim_{n'' \to +\infty} \s^{2}_{n''}(\eps) &= \s^2 M_\infty^{(2)} , \\
    \label{condition3-bis}
    \lim_{n'' \to +\infty} \eta_{n''} &= 0
  \end{align}
  are simultaneously satisfied.
  
  First of all note that since
  \[
  y^2(1-F_0(y)) \leq \int_{(y,+\infty)}x^2dF_0(x) \quad \text{and}
  \quad y^2 F_0(-y) \leq   \int_{(-\infty,-y]}x^2dF_0(x) \quad (y>0)
  \]
  and $\int_\RE x^2dF_0(x)<+\infty$, 
  it follows that $\lim_{y \to +\infty } y^2(1-F_0(y)) = \lim_{y \to -\infty } y^2(F_0(y)) =0$. 
  Hence, given $\eps>0$, there exists a $Y=Y(\eps)$ 
  such that $y^2(1-F_0(y))<\eps$ for every $y>Y$. 
  Since
  \[
  \z_{n''}(x)=\sum_{j=1}^{n''}(\b_{j,n''}^{-1}x)^2
  (1-F_0(\b_{j,n''}^{-1}x)) \b_{j,n''}^2/x^2 \qquad (x>0),
  \]
  one gets
  \[
  \z_{n''}(x) \leq \eps M_{n''}^{(2)}\frac{1}{x^2}
  \]
  whenever $x>\beta_{(n'')}Y$. 
  In view of property \eqref{conMn},
  the first relation (\ref{condition1-uno-bis}) follows for $x>0$.
  The argument for $x<0$ is analogous.

  We turn to the proof of (\ref{clt3condtris}).
  A simple computation reveals
  \[
  s_{n''}^2(\eps)
  :=\sum_{j=1}^{n''} \int_{(-\eps,\eps]} x^2 dQ_{j,n''}(x)
  =\s^2M_{n''}^{(2)}- R_{n''} ,
  \]
  with the remainder term 
  \[ 
  R_{n''}
  :=\sum_{j=1}^{n''} \b_{j,n''}^2 \int_{|\b_{j,n''} x| >\eps} x^2 dF_0(x)
  \leq M_{n''}^{(2)} \int_{|\b_{(n'')} x| >\eps} x^2 dF_0(x).
  \]
  Invoking property \eqref{conMn} again, 
  it follows that $R_{n''}\to0$ as $n''\to\infty$; 
  recall that $F_0$ has finite second moment by hypothesis.
  Consequently,
  \begin{equation}\label{eq:con-gaus2-a}
    \lim_{n'' \to +\infty} s^{2}_{n''}(\eps) = \s^2 M_\infty^{(2)}
  \end{equation}
  for every $\eps$. 
  Since $\int_\RE x \,dQ_{j,n}(x)=0$,
  \[
  \sum_{j=1}^{n''} \Big(\int_{(-\eps,\eps]} x \,dQ_{j,n}(x)\Big )^2
  \leq \sum_{j=1}^{n''} \b_{j,n''}^2 \Big( \int_{|\b_{j,n''} x|\geq\eps} x dF_0(x) \Big)^2
  \leq \Big( \int_{|\b_{(n'')} x|\geq\eps} |x| dF_0(x) \Big)^2 \sum_{j=1}^{n''} \b_{j,n''}^2 ,
  \]
  which yields that
  \[
  \lim_{n'' \to +\infty} \sum_{j=1}^{n''} 
  \Big(\int_{(-\eps,\eps]} x \,dQ_{j,n}(x)\Big )^2=0.
  \]
  Combining this last fact with (\ref{eq:con-gaus2-a}) gives (\ref{clt3condtris}).

  Finally, in order to obtain (\ref{condition3-bis}),
  we use \eqref{condition1-uno-bis} and
  the dominated convergence theorem;
  the argument is the same as for \eqref{eq.etafromzeta}
  in the proof of Lemma \ref{lemma4_2}.
\end{proof}

\begin{proof}[Proof of Theorem \ref{thm3-bis}.]
  Use Lemma \ref{lemma4_5} and repeat the proof of Theorem \ref{thm2}.
  A trivial adaptation is needed in the calculation of moments if $\gamma>2$:
  consider \eqref{eq.momcalc} with $G_\a=G_2$, the distribution function of a 
  Gaussian law, and note that it posses finite moments of every order.
  Hence $\E[|V_\infty|^p]$ is finite if and only if $\E[(M^{(2)}_\infty)^p]$ is finite,
  which, by Lemma \ref{Lemma2bis}, is the case if and only if $\QQ(p)<0$.
\end{proof}

\subsection{Proof of convergence for $\alpha=1$ (Theorems \ref{thm1} and \ref{thm2-bis})}
Let us first prove Theorem \ref{thm2-bis}.
We shall apply the central limit theorem to the random variables
\begin{align*}
  W_n^* = \sum_{j=1}^n (\b_{j,n}X_j-q_{j,n})
\end{align*}
with $q_{j,n}$ defined in (\ref{centr-case1}).  
In what follows,
\begin{align*}
  Q_{j,n}(x) := F_0 \Big( \frac{x+q_{j,n}}{\b_{j,n}} \Big).
\end{align*}
The next Lemma is the analogue of Lemma \ref{lemma3} above.

\begin{lemma}
  \label{lemma3alfa=1} 
  Suppose the assumptions of Theorem 
  \ref{thm2-bis} are in force.
  Then, for every $\delta \in (0,1)$, 
  \begin{equation}\label{CENTRATURA}
    |q_{j,n}|= \Big| \int \sin(\beta_{j,n}s)dF_0(s) \Big| \leq C_{\delta}\beta_{j,n}^{1-\delta}
  \end{equation}
  with $C_\delta=\int_\RE |x|^{1-\d}dF_0(x) <+\infty$.
  Furthermore, for every divergent sequence $(n')$ of integer numbers,
  there exists a divergent subsequence
  $(n'') \subset (n')$ and a set $\Omega_0$ of probability one 
  such that for every $\omega$ in $\Omega_0$ and for every $\eps>0$,
  the properties \eqref{conMn} and \eqref{UAN} are verified.
\end{lemma}

\begin{proof}
  First of all note that $C_\delta <+\infty$ for every $\delta\in(0,1)$
  because of hypothesis (\ref{stabledomain}).
  Using further that $ |\sin(x)| \leq |x|^{1-\delta}$ for $\delta\in(0,1)$,
  one immediately gets
  \[
  | q_{j,n}| = \left |
    \int_\RE \sin(\beta_{j,n}s) dF_0(s)
  \right | 
  \leq  \beta_{j,n}^{1-\delta } \int_\RE |s|^{1-\delta}dF_0(s).
  \]

  To prove \eqref{UAN} note that, as a consequence of \eqref{CENTRATURA},
  \begin{align*}
    \frac{\epsilon+q_{j,n}}{\beta_{j,n}} 
  \geq \beta_{(n)}^{-1}( \epsilon - C_\delta \beta_{(n)}^{1-\delta} ) .
  \end{align*}
  Clearly, the expression inside the bracket is positive for sufficiently small $\beta_{(n)}$.
  Defining $(n'')$ and $\Omega_0$ in accordance to Lemma \ref{lemma3},
  it thus follows
  \[
  \max_{1\leq j \leq n''} \Big(1-Q_{j,n''}(\eps)+Q_{j,n''}(-\eps)\Big )
  \leq 1-F_0 (\bar{c}\beta_{(n'')}^{-1}) + F_0 (-\bar{c}\beta_{(n'')}^{-1})
  \]
  for a suitable constant $ \bar{c}$ depending only on $\eps$, $\delta$ and $F_0$.
  An application of (\ref{conMn}) yields \eqref{UAN}.
\end{proof}

\begin{lemma}\label{lemma4_4} 
  Suppose the assumptions of Theorem 
  \ref{thm2-bis} are in force,
  then for every divergent sequence $(n')$ of integer numbers
  there exists a divergent subsequence $(n'') \subset (n')$ 
  and a measurable set $\Omega_0$ with $P(\Omega_0)=1$ such that
  \begin{align}
    \label{eq.limitalpha1}
    \lim_{n'' \to +\infty} \E[ e^{ i \xi W^*_{n''}}|\CB](\omega)=\exp\{ -|\xi|
    k_1 M_{\infty}^{(1)}(\omega) (1+i2 \eta \log|\xi|\operatorname{sign}\xi)   \}
    \quad (\xi \in \RE )
  \end{align}
  for every $\omega$ in $\Omega_0$.
\end{lemma}

\begin{proof}
  Define $(n'')$ and $\Omega_0$ according to Lemma \ref{lemma3alfa=1},
  implying the convergencies \eqref{conMn}, 
  and the UAN condition \eqref{UAN}.
  In the following, let $\omega\in\Omega_0$ be fixed.
  In view of Proposition 11 in Section 17.3 of \cite{fristedgray}
  the claim \eqref{eq.limitalpha1} follows if 
  \eqref{condition1-uno}, \eqref{condition1-due} and \eqref{clt3condBis} 
  are satisfied with $\alpha=1$, 
  and in addition
  \begin{equation}\label{condition3_bis}
    \lim_{n\to +\infty} \sum_{j=1}^{n}\int_\RE \chi(t) dQ_{j,n''}(t)
    = M_{\infty}^{(1)}(c^+-c^-)\int_0^\infty \frac{\chi(t)- \sin(t)}{t^2}dt
  \end{equation}
  with $\chi(t)=-\J\{t \leq -1 \}+t \J\{ -1 <t <1 \}+\J\{t \geq 1 \}$. 

  Let us verify \eqref{condition1-uno} for an arbitrary $x>0$.
  Given $\epsilon>0$, there exists some $Y=Y(\epsilon)$ such that
  \begin{align}
    \label{eq.yinc}
    c^+ -\eps \leq y(1-F_0(y)) \leq c^+  +\eps
  \end{align}
  for all $y\geq Y$ because of hypothesis \eqref{stabledomain}.
  Moreover, in view in Lemma \ref{lemma3alfa=1},
  \begin{align*}
    \hat y_{j,n''}:=\frac{x+q_{j,n''}}{\b_{j,n''}} \geq \frac{x-C_{1/2}\b_{(n'')}^{1/2}}{\b_{(n'')}},
  \end{align*}
  which clearly diverges to $+\infty$ as $n''\to\infty$ because of \eqref{conMn};
  in particular, $\hat y_{j,n''}\geq Y$ for $n''$ large enough.
  It follows by \eqref{eq.yinc} that for those $n''$,
  \begin{align}
    \label{eq.ybound}
    \frac{c^+ -\eps}{x+q_{j,n''}} \b_{j,n''} \leq 1-F(\hat y_{j,n''}) \leq \frac{c^+ +\eps}{x+q_{j,n''}} \b_{j,n''} .
  \end{align}
  Recalling that $\zeta_{n''}(x)=\sum_{j=1}^{n''} [1-F(\hat y_{j,n''})]$,
  summation of \eqref{eq.ybound} over $j=1,\ldots,n''$ gives
  \begin{align*}
    \frac{c^+ -\eps}{x+q_{j,n''}} M^{(1)}_{n''} \leq \zeta_{n''}(x) \leq 
   \frac{c^+ +\eps}{x+q_{j,n''}} M^{(1)}_{n''} .
  \end{align*}
  Finally, observe that $|q_{j,n''}|\leq C_{1/2}\b_{(n'')}^{1/2}\to0$ as $n''\to\infty$,
  and that $M^{(1)}_{n''}\to M^{(1)}_\infty$ by \eqref{conMn}.
  Since $\eps>0$ has been arbitrary, the claim \eqref{condition1-uno} follows.
  The proof of \eqref{condition1-due} for arbitrary $x<0$ is completely analogous.

  Concerning \eqref{clt3condBis}, it is obviously enough to prove that
   \begin{align}
    \label{eq.sconv}
    \lim_{\eps \to 0}\limsup_{n'' \to +\infty}s^2_{n''}(\eps)=0
  \end{align}
  where 
  \(
  s^2_{n''}(\eps):= \sum_{j=1}^{n''} \int_{(-\eps,\eps]}x^2 dQ_{j,n''}(x)
  \).
  We split the domain of integration in the definition of $s^2_{n''}$ at $x=0$, 
  and integrate by parts to get
  \begin{align*}
    s^2_{n''}(\eps)
    = & -\eps^2 \sum_{j=1}^{n''} Q_{j,n''}(-\eps) - \sum_{j=1}^{n''} \int_{(-\eps,0]}Q_{j,n''}(u)2u du  \\
    & -\eps^2  \sum_{j=1}^{n''}(1- Q_{j,n''}(\eps)) +\sum_{j=1}^{n''}\int_{(0,\eps]}(1- Q_{j,n''}(u))2u du  \\
    =:& A_{n''}(\eps)+B_{n''}(\eps)+C_{n''}(\eps)+D_{n''}(\eps). \\
  \end{align*}
  Having already proven (\ref{condition1-uno}) and (\ref{condition1-due}),
  we conclude
  \begin{equation}
    \label{due_bis}
    \lim_{\eps \to 0^+}\lim_{n'' \to +\infty}\{ |A_{n''}(\eps)| + |C_{n''}(\eps)|\}=0.
  \end{equation}
  Fix $\eps>0$; assume that $n''$ is sufficiently large to have $|q_{j,n''}|< \eps/2$ for $j=1,\dots,n''$. 
  Then
  \begin{align*}
    |B_n(\epsilon)| 
    \leq & \sum_{j=1}^{n''} 2 \int_0^\eps w F_0\left(\frac{-w+q_{j,n''}}{\b_{j,n''}}\right)dw \\
    \leq & \sum_{j=1}^{n''} \left\{ \int_0^{2|q_{j,n''}|} 2w\,dw 
      + 2 \int_{2|q_{j,n''}|}^{\eps} wF_0\left(\frac{-w+q_{j,n''}}{\b_{j,n''}}\right)dw\right\} \\
    \leq & \sum_{j=1}^{n''} \Big\{ 4 |q_{j,n''}|^2 
    + 2\int_{2|q_{j,n''}|}^{\eps} w \Big(\frac{K \b_{j,n''}}{w-q_{j,n''}}\Big)dw \Big\} \\
    \leq &  \sum_{j=1}^{n''} \Big\{ 4 C_{1/4}^2 \b_{j,n''}^{3/2} + \b_{j,n''} \int_0^\eps 4K\,dw \Big\}
    \leq \big( 4 C_{1/4}^2 \b_{(n'')}^{1/2} + 4K\eps \big) M^{(1)}_{n''} ,
  \end{align*}
  with the constant $K$ defined in \eqref{eq.defK}.
  In view of \eqref{conMn}, it follows
  \begin{equation}\label{tre_bis}
    \lim_{\eps \to 0^+}\limsup_{n'' \to +\infty}|B_{n''}(\eps)|=0
  \end{equation}
  as desired.
  A completely analogous reasoning applies to $D_{n''}$.
  At this stage we can conclude \eqref{eq.sconv},
  and thus also \eqref{clt3condBis}.

  In order to verify  (\ref{condition3_bis}), 
  let us first show that
  \begin{equation}
    \label{eq:0001}
    \lim_{n''\to\infty}\sum_{j=1}^{n''}\int_\RE \sin(x) \,dQ_{j,n''}(x) = 0.
  \end{equation}
  We find
  \begin{align*}
    & \Big |\sum_{j=1}^{n''}   \int_\RE \sin(x)dQ_{j,n''}(x)\Big | 
    = \Big | \sum_{j=1}^{n''} \int_\RE \sin(t\b_{j,n''}-q_{j,n''})dF_0(t) \Big  | \\
    & \qquad \leq \sum_{j=1}^{n''} \Big| \cos(q_{j,n''})\int_\RE \sin(t\b_{j,n''})dF_0(t)
    - \sin(q_{j,n''})\int_\RE \cos(t\b_{j,n''})dF_0(t) \Big| \\
    & \qquad = \sum_{j=1}^{n''} \Big| (\cos(q_{j,n''})-1)q_{j,n''} + (q_{j,n''}-\sin(q_{j,n''}))
    + \sin(q_{j,n''})\int_\RE (1-\cos(t\b_{j,n''}))dF_0(t) \Big  | \\
    & \qquad \leq \sum_{j=1}^{n''} \big( | I_1 | + | I_2 | + | I_3 | \big).
  \end{align*}
  The elementary inequalities $|\cos(x)-1|\leq x^2/2$ and $|x-\sin(x)| \leq x^3/6$
  provide the estimate
  \[
  \sum_{j=1}^{n''} \big( |I_1| + |I_2| \big) 
  \leq \sum_{j=1}^{n''} |q_{j,n''}|^3 \leq C_{1/2}^3 \b_{(n'')}^{1/2} M^{(1)}_{n''}.
  \]
  By \eqref{conMn}, the last expression converges to zero as $n''\to\infty$.
  In order to estimate $I_3$, observe that, since $|1-\cos(x)| \leq 2 x^{3/4}$
  for all $x\in\RE$,
  \[
  \int_\RE (1- \cos(t\b_{j,n''}))dF_0(t)
  \leq 2 \b_{j,n''}^{3/4} \int_\RE |t|^{3/4} dF_0(dt) 
  = 2 C_{1/4}\b_{j,n''}^{3/4} .
  \]
Consequently, applying Lemma \ref{lemma3alfa=1} once again,
  \begin{align*}
    \sum_{j=1}^{n''} | I_3 | 
    \leq \sum_{j=1}^{n''} |q_{j,n''}| \Big| \int_\RE (1-\cos(t\b_{j,n''}))dF_0(t) \Big|
    \leq 2C_{1/4}^2 \b_{(n'')}^{1/2} M^{(1)}_{n''} ,
  \end{align*}
  which converges to zero on grounds of \eqref{conMn}.

  Having proven \eqref{eq:0001}, 
  the condition \eqref{condition3_bis} becomes equivalent to 
  \begin{align}
    \label{condition3_tris}
    \sum_{j=1}^{n''}\int_\RE (\chi(t)-\sin(t))dQ_{j,n''}(t) \to 
    M_{\infty}^{(1)}(c^+-c^-)\int_{\Re^+} \frac{\chi(t)- \sin(t)}{t^2}dt .
  \end{align}
  The proof of this fact follows essentially the line 
  of the proof of Theorem 12 of \cite{fristedgray}.
  Let us first prove that, if $-\infty<x<0<y<+\infty$,
  \begin{equation}\label{step1}
    \lim_{n''\to +\infty} \int_{(x,y]} d\nu_{n''}(t) = M_\infty^{(1)}(c^+y-c^-x), 
  \end{equation}
  where the sequence $(\nu_n)$ of measures on $\RE$ is defined by
  \[
  \nu_n[B] = \sum_{j=1}^{n} \int_B t^2 dQ_{j,n}(t) 
  \]
  for Borel sets $B\subset\RE$.
  For fixed $\eps\in(0,y)$, one uses \eqref{condition1-uno} to conclude
  \begin{align*}
    \lim_{n''\to\infty} \int_{(\eps,y]} d\nu_{n''}(t) 
    & = \lim_{n''\to\infty} \bigg( t^2\sum_{j=1}^{n''}\big(1-Q_{j,n''}(t)\big)\Big|^\eps_y + 2\int_{(\eps,y]} t \sum_{j=1}^{n''}\big(1-Q_{j,n''}(t)\big) dt \bigg) \\
    & = \eps^2\frac{c^+M_\infty^{(1)}}{\eps} - y^2\frac{c^+M_\infty^{(1)}}{y} + 2 \int_{(\eps,y]} t \frac{c^+M_\infty^{(1)}}{t}dt \\
    & = (y-\epsilon)c^+M_\infty^{(1)} .
  \end{align*}
  Notice that we have used the dominated convergence theorem to pass to the limit
  under the integral;
  this is justified in view of the upper bound provided by \eqref{eq.thirtyeightandonehalf}.
  In a similar way, one shows for fixed $\eps\in(0,|x|)$ that
  \begin{align*}
    \lim_{n''\to\infty} \int_{(x,-\eps]} d\nu_{n''}(t) = ( |x| - \eps)c^-M_\infty^{(1)} .
  \end{align*}
  Combining this with \eqref{eq.sconv}, one concludes
  \begin{align*}
    & \limsup_{n''\to\infty} \int_{(x,y]} d\nu_{n''}(t) \\
    & \quad \leq \limsup_{\eps\to0}\limsup_{n''\to\infty} \Big( \int_{(x,-\eps]} d\nu_{n''}(t) 
    + \int_{(-\eps,\eps]} d\nu_{n''}(t) + \int_{(\eps,y]} d\nu_{n''}(t) \Big) \\
    & \quad \leq \lim_{\eps\to0}(y-\eps)c^+M_\infty^{(1)} - \lim_{\eps\to0}(x + \eps)c^-M_\infty^{(1)} + 
    \lim_{\eps\to0}\limsup_{n''\to\infty} \sum_{j=1}^{n''}\int_{(-\eps,\eps]} t^2dQ_{j,n''}(t) \\
    & \quad = (c^+y-c^-x)M_\infty^{(1)} .
  \end{align*}
  On the other hand, trivially,
  \begin{align*}
    \liminf_{n''\to\infty} \int_{(x,y]} d\nu_{n''}(t)
    \geq \liminf_{\eps\to0}\liminf_{n''\to\infty} \Big( \int_{(x,-\eps]} d\nu_{n''}(t) + \int_{(\eps,y]} d\nu_{n''}(t) \Big)
    = (c^+y-c^-x)M_\infty^{(1)} .
  \end{align*}
  This proves \eqref{step1}.
  Now fix $0<R<+\infty$, and note that \eqref{step1} yields that 
  for every bounded and continuous function $f:[-R,R]\to\RE$
  \begin{align}
    \label{step3}
    \lim_{n''\to\infty} \int_{[-R,R]}  f(t) d\nu_{n''}(t) = 
    M_\infty^{(1)} c^- \int_{-R}^0 f(t)dt + M_\infty^{(1)} c^+ \int_0^R f(t)dt 
  \end{align}
  holds true. 
  In particular, using $f(t) = (\chi(t)-\sin t)/t^2$, for every $0<R<+\infty$, one gets 
 \begin{equation}\label{step4}
  \begin{split}
        \lim_{n''\to\infty} \sum_{j=1}^{n''} &\int_{[-R,R]}  (\chi(t)-\sin(t))
     dQ_{j,n''}(t)\\ &= 
 M_\infty^{(1)} (c^+-c^-) \int_{0}^R \frac{(\chi(t)-\sin(t))}{t^2}  dt.
  \\
  \end{split}
  \end{equation}
  Moreover, since $|\chi(t)-\sin t|\leq 2$, 
  \[
  |\sum_{j=1}^{n''} \int_{[-R,R]^c}  (\chi(t)-\sin(t))
     dQ_{j,n''}(t)| \leq  2 [\z_{n''}(-R)+\z_{n''}(R)].
  \]
 Applying  \eqref{condition1-uno}
  and \eqref{condition1-due} one obtains
  \[
  \limsup_{n'' \to +\infty}|\sum_{j=1}^{n''} \int_{[-R,R]^c}  (\chi(t)-\sin(t))
     dQ_{j,n''}(t)| \leq 2 M_\infty^{(\a)}(c^+ + c^-) \frac{1}{R},
  \]
  which gives
   \begin{align}
    \label{step5}
 \limsup_{R \to + \infty} \limsup_{n'' \to + \infty} | \sum_{j=1}^{n''} 
  \int_{[-R,R]^c}  (\chi(t)-\sin(t))    dQ_{j,n''}(t)|=0.
\end{align}  
Combining  \eqref{step4} with \eqref{step5} one gets \eqref{condition3_tris}.
\end{proof}

\begin{proof}[Proof of Theorem \ref{thm2-bis}]
  Use Lemma \ref{lemma4_4} and repeat the proof of Theorem \ref{thm2}.
\end{proof}

\begin{proof}[Proof of Theorem \ref{thm1}]
  The theorem is a corollary of Theorem \ref{thm2-bis}. 
  Since $m_0=\int_{\RE} x\, dF_0(x)<\infty$ by hypothesis, it follows that $c^+=c^-=0$,
  and so $V^*_t$ converges to $0$ in probability.
  Now write
  \[
  V_t=m_0M_{\n_t}^{(1)} + V^*_t-R_{\n_t} ,
  \]
  with the remainder
  \[
  R_n:=\sum_{j=1}^n(q_{j,n}-\b_{j,n}m_0).
  \]
  Thanks to Lemma \ref{Lemma2bis}, 
  $m_0M_{\n_t}^{(1)}$ converges in distribution to $m_0M_\infty^{(1)}$. 
  It remains to prove that $R_{\nu_t}$ converges to $0$ in probability.
  Since
  \[
  \Big|\frac{\sin(x)}{x}-1\Big| \leq H(x) := 1/6 \big[x^2\J\{|x|<1\}+\J\{|x|\geq1\} \big] \leq 1/6,
  \]
  it follows that
  \begin{align*}
    |R_n| 
    &\leq \sum_{j=1}^n \b_{j,n} \int_\RE \Big|\frac{\sin(\b_{j,n}x)}{\b_{j,n}x}-1\Big||x|dF_0(x) \\
    & \leq \sum_{j=1}^n \b_{j,n} \int_\RE H(\b_{j,n}x)|x|dF_0(x) 
    \leq M_n^{(1)} \int_\RE H(\b_{(n)}x)|x|dF_0(x).
  \end{align*}
  Recall that $M_{n}^{(1)}$ converges a.s. to $M_\infty^{(1)}$ 
  and $\b_{(n)}$ converges in probability to $0$ by \eqref{conMn}. 
  By dominated convergence it follows that also $\int_\RE H(\b_{(n)}x)|x|dF_0(x)$
  converges in probability to $0$. 

  The (non-)degeneracy of $V_\infty$ and the (in)finiteness of its moments
  is an immediate consequence of Lemma \ref{Lemma2bis}.
\end{proof}

\subsection{Estimates in Wasserstein metric (Proposition \ref{PropW-2})} 

\begin{proof}[Proof of Proposition \ref{PropW-2}]
  The proof uses the techniques employed in Lemma \ref{Lemma2bis}.

  We shall assume that $\W_\gamma(X_0,V_\infty)<+\infty$, since otherwise the claim is trivial.
  Then, there exists an optimal pair $(X^*,Y^*)$ realizing the infimum
  in the definition of the Wasserstein distance,
  \begin{align}
    \label{eq.defdel}
    \Delta:=\W_{\gamma}^{\max(\gamma,1)}(X_0,V_{\infty})=\W_\gamma^{\max(\gamma,1)}(X^*,Y^*)=\E|X^* -Y^* |^\gamma.
  \end{align}
  Let $(X_j^*,Y_j^*)_{j\geq1}$ be a sequence of independent and identically distributed random variables
 with the same law of $(X^*,Y^*)$,
  which are further independent of $B_n=(\b_{1,1},\b_{1,2},\dots,\b_{n,n})$.
  Consequently, $\sum_{j=1}^n X_j^* \b_{j,n}$ has the same law of $W_n$,
  and $\sum_{j=1}^n Y_j^* \b_{j,n}$ has the same law of $V_\infty$.
  By definition of $\W_\gamma$,  
  \begin{align*}
    \W_\gamma^{\max(\gamma,1)}  (W_n,V_\infty) 
    & \leq \E \Big[ \Big |\sum_{j=1}^{n} X_{j}^*  \b_{j,n} - \sum_{j=1}^{n} Y_{j}^* \b_{j,n} \Big|^\gamma \Big] \\
    &=\E \Big[ \E \bigg[ \Big | \sum_{j=1}^{n} (X_{j}^*-Y_j^*)  \b_{j,n} \Big|^\gamma \bigg|B_n \bigg] \Big]. \\
  \end{align*}
  For further estimates, we distinguish two cases.
  In the first case, $0<\a<\gamma\leq1$, we apply the elementary inequality
  \begin{align*}
    \bigg| \sum_{j=1}^n z_j \bigg|^\gamma \leq \sum_{j=1}^n |z_j|^\gamma
  \end{align*}
  for real numbers $z_1,\ldots,z_n$ to obtain
  \begin{align*}
    \W_\gamma(W_n,V_\infty) 
    \leq \E\Big[ \E\bigg[ \sum_{j=1}^n \b_{j,n}^\gamma |X_j^*-Y_j^*|^\gamma \bigg| B_n \bigg] \Big] 
    = \E\Big[ \sum_{j=1}^n \b_{j,n}^\gamma \Big] \Delta ;
  \end{align*}
  where  $\Delta$ is defined in \eqref{eq.defdel}.
  In the second case, $1\leq \a < \gamma\leq2$, we 
  can apply the Bahr-Esseen inequality (\ref{bahr-essen2})
  since $E(X^*_j-Y_j^*)=E(X_1)-E(V_\infty)=0$ and $E|X^*_j-Y_j^*|^\g  = W_\g^\g(X_0,V_\infty)<+\infty$.
  Thus,
  \begin{align*}
    \W_\gamma^\gamma(W_n,V_\infty)
    \leq \E\Big[ \bigg[ 2 \sum_{j=1}^{n} \b_{j,n}^{\gamma} | X_{j}^* -Y_{j}^* |^\gamma \bigg| F_n \Big] 
    = 2 \E\big[ \sum_{j=1}^{n} \b_{j,n}^{\gamma} \Big] \Delta .
  \end{align*}
  By convexity of the Wasserstein metric, 
  \[
  \W_\gamma^{\max(\gamma,1)}  (V_t,V_{\infty}) \leq \sum_{n \geq 1} e^{-t}(1-e^{-t})^{n-1} \W_\gamma^{\max(\gamma,1)}(W_n,V_\infty) .
  \]
  Combining the previous estimates with Proposition \ref{PropConst},
  we obtain
  \begin{align*}
    \W_\gamma^{\max(\gamma,1)}  (V_t,V_{\infty}) 
    & \leq a \Delta \sum_{n \geq 1} e^{-t}(1-e^{-t})^{n-1} \E \Big[ \sum_{j=1}^{n} \b_{j,n}^{\gamma}\Big] \\
    & = a \Delta e^{t\QQ (\gamma)},
  \end{align*}
  with $a=1$ if $0<\a <\gamma\leq 1$ and $a=2$ if $1\leq \a < \gamma \leq 2$.
\end{proof}

Lemma \ref{lemma-tail2} is a corollary of the following.
\begin{lemma}
  \label{lemma-tail1}
  Let two random variables $X_1$ and $X_2$ be given,
  and assume that their distribution functions $F_1$ and $F_2$
  both satisfy the conditions \eqref{eq.Ftail} and \eqref{eq.Ftail2}
  with the same constants $\alpha>0$, $0<\eps<1$, $K$ and $c^+,c^-\geq0$.
  Then $\W_\gamma(X_1,X_2)<\infty$ for all $\gamma$ that satisfy $\alpha<\gamma<\frac{\alpha}{1-\epsilon}$.
\end{lemma}
\begin{proof}
  Define the auxiliary functions $H$, $H_+$ and $H_-$ on $\RE\setminus\{0\}$ by
  \begin{align*}
    H(x) = \J\{x>0\}(1-c^+x^{-\a}) + \J\{x<0\}c^-|x|^{-\a} , \quad
    H_±(x) = H(x) ± K |x|^{-(\a+\epsilon)},
  \end{align*}
  so that $H_-\leq F_i\leq H_+$ for $i=1,2$ by hypothesis.
  It is immediately seen that $H(x)$, $H_+(x)$ and $H_-(x)$ all tend to one  (to zero, 
  respectively)  when $x$ goes to $+\infty$ ($-\infty$, respectively).
  Moreover, evaluating the functions' derivatives,
  one verifies that $H$ and $H_-$ are strictly increasing on $\RE_+$,
  and that $H_+$ is strictly increasing on some interval $(R_+,+\infty)$.
  Let $\check R>0$ be such that $H(\check R)>H_+(R_+)$;
  then, for every $x>\check R$, the equation
  \begin{align}
    \label{eq.hhh}
    H_-(\hat x) = H(x) = H_+( \check x)
  \end{align}
  possesses precisely one solution pair $(\check x,\hat x)$ satisfying $R_+<\check x<x<\hat x$.
  Likewise, $H$ and $H_+$ are strictly increasing on $\RE_-$,
  and $H_-$ is strictly increasing on $(-\infty,-R_-)$.
  Choosing $\hat R>0$ such that $H(-\hat R)<H_-(-R_-)$,
  equation \eqref{eq.hhh} has exactly one solution $(\check x,\hat x)$ 
  with $\check x<x<\hat x<-R_-$ for every $x< -\hat R$.
  
  A well-known representation of the Wasserstein distance of measures on $\RE$ reads
  \begin{align*}
    \W_\gamma^{\max(\gamma,1)}(X_1,X_2) = \int_0^1 | F_1^{-1}(y) - F_2^{-1}(y) |^\gamma \,dy ,
  \end{align*}
  where $F_i^{-1}:(0,1)\to\RE$ denotes the pseudo-inverse function of $F_i$.
  We split the domain of integration $(0,1)$ into the 
  three intervals $(0,H(-\hat R))$, $[H(-\hat R),H(\check R)]$ and $(H(\check R),1)$,
  obtaining:
  \begin{align*}
    \W_\gamma(X_1,X_2)^{\max(\gamma,1)}
    &= \int_{-\infty}^{-\hat R} | F_1^{-1}(H(x))-F_2^{-1}(H(x))|^\gamma H'(x)\,dx \\
    & + \int_{H(-\hat R)}^{H(\check R)} | F_1^{-1}(y) - F_2^{-1}(y) |^\gamma \,dy \\
    & + \int_{\check R}^\infty | F_1^{-1}(H(x))-F_2^{-1}(H(x))|^\gamma H'(x)\,dx .
  \end{align*}
  The middle integral is obviously finite.
  To prove finiteness of the first and the last integral,
  we show that
  \begin{align*}
    \int_{\check R}^\infty | F_1^{-1}(H(x)) - x |^\gamma H'(x)\,dx < \infty ;
  \end{align*}
  the estimates for the remaining contributions are similar.
  Let some $x\geq \check R$ be given, 
  and let $\hat x>\check x>\check R$ satisfy \eqref{eq.hhh}.
  From $H_-<F_1<H_+$, it follows that
  \begin{align*}
    F_1(\check x) < H(x) < F_1(\hat x),
  \end{align*}
  which implies further that
  \begin{align}
    \label{eq.xhx}
    \check x - x < F_1^{-1}(H(x)) - x < \hat x - x .
  \end{align}
  From the definition of $H$, it follows that
  \begin{align*}
    x = \hat x (1+\kappa\hat x^{-\eps})^{-1/\alpha} ,
  \end{align*}
  with $\kappa=K/c^+$.
  Combining this with a Taylor expansion, and recalling that $\hat x>x>\check R>0$, 
  one obtains
  \begin{align}
    \label{eq.xhat}
    \hat x - x = x \big[(1+\kappa\hat x^{-\eps})^{1/ \a} - 1 \big] 
    < x \big[ (1+\kappa x^{-\eps})^{1/ \a}-1\big] < \hat C x^{1-\eps} ,
  \end{align}
  where $\hat C$ is defined in terms of $\alpha$, $\kappa$ and $\check R$.
  In an analogous manner,
  one concludes from
  \begin{align*}
    x = \check x (1-\kappa\check x^{-\eps})^{-1/\alpha} ,
  \end{align*}
  in combination with $0<R_+<\check x<x$ and $0<\eps<1$ that
  \begin{align}
    \label{eq.xcheck}
    \check x - x = \check x\big[1-(1-\kappa\check x^{-\eps})^{-1/ \a}\big]
    \geq \check x\big[1-(1+\check C \check x^{-\eps})\big] 
    = - \check C \check x^{1-\eps} > - \check C x^{1-\eps},
  \end{align}
  where $\check C$ only depends on $\a$, $\kappa$ and $R_+$.
  Substitution of \eqref{eq.xhat} and \eqref{eq.xcheck} into \eqref{eq.xhx} yields
  \begin{align*}
    \int_{\check R}^\infty | F_1^{-1}(H(x)) - x |^\gamma H'(x)\,dx 
    < \max(\hat{C},\check{C})^\gamma \int_{\check R}^\infty x^{\gamma(1-\eps)-\alpha-1}\,dx ,
  \end{align*}
  which is finite provided that $0<\gamma<\alpha/(1-\eps)$.
\end{proof}
\begin{proof}[Proof of Lemma \ref{lemma-tail2}]
  In view of Lemma \ref{lemma-tail1}, 
  it suffices to show that the distribution function $F_\infty$ of $V_\infty$ 
  satisfies \eqref{eq.Ftail} and \eqref{eq.Ftail2}
  with the same constants $c^+$ and $c^-$ as the initial condition $F_0$
  (possibly after diminishing $\eps$ and enlarging $K$).

  The proof is based on the representation of $F_\infty$ as a mixture of stable laws.
  More precisely, let $G_\a$ be the distribution function 
  whose characteristic function is $\hat g_\a$ as in \eqref{chaSta},
  then
  \begin{align*}
    F_\infty (x) = \E\Big[ G_\a \Big( \big(M_\infty^{(\alpha)}\big)^{-1/\alpha} x \Big) \Big] ,
  \end{align*}
  see \eqref{characteristic}.
  Since $\a<\gamma<2\a$, then 
  there exists a finite constant $K>0$ such that
  \begin{align*}
    | 1 - c_+x^{-\alpha} - G_\a(x) | \leq K x^{-\gamma}
  \end{align*}
  for $x>0$, and similarly for $x<0$; 
  see, e.g. Sections 2.4 and 2.5 of \cite{Zolotarev1986}).
  Using that $\E[M_\infty^{(\alpha)}]=1$ and $C:=\E[(M_\infty^{(\alpha)})^{\gamma/\a}]<\infty$ 
  (since $\QQ(\gamma)<0$) it follows further that
  \begin{align*}
    \big| 1 - c^+ x^{-\alpha} - F_\infty(x) \big|
    & = \big| 1 - c^+ \E\big[M_\infty^{(\alpha)}\big]x^{-\alpha} - \E \big[G_\a ((M_\infty^{(\alpha)})^{-1/\alpha}x) 
         \big] \big| \\
    & \leq \E \big[ \big| 1 - c^+ \big( (M_\infty^{(\alpha)})^{-1/\alpha} 
    x\big)^{-\alpha} -  G_\a ((M_\infty^{(\alpha)})^{-1/\alpha}x)\big| \big] \\
    & \leq \E \big[ K (M_\infty^{(\alpha)})^{\gamma/\a} x^{-\gamma} \big] = C K x^{-\gamma} .
  \end{align*}
  This proves \eqref{eq.Ftail} for $F_\infty$, with $\eps=\gamma-\a$ and $K'=CK$.
  A similar argument proves \eqref{eq.Ftail2}.
\end{proof}

\subsection{Proofs of strong convergence (Theorem \ref{thm.strongnew})}

We shall use the Wild sum representation of the solution to the Boltzmann equation,
see \eqref{eq.wildsum} and \eqref{eq.wildrec}.
The idea is to prove that certain $\xi$-pointwise a priori bounds 
on the characteristic functions $\qn$ are preserved by the collisional operator,
and hence are propagated from the initial condition to any later time.

A first intermediate result is
\begin{lemma}
  \label{lem.head} Under the hypotheses of Theorem \ref{thm.strongnew}, 
there exists a constant $\theta>0$ and a radius $\rho>0$, both independent of $n\geq0$, 
  such that $|\qn(\xi)|\leq(1+\theta|\xi|^\alpha)^{-1/r}$ for all $|\xi|\leq\rho$.
\end{lemma}
\begin{proof}
  By the explicit representation 
  \eqref{characteristic} or \eqref{characteristic-gauss}, 
  respectively,
  we conclude that 
  \begin{align*}
    |\phi_\infty(\xi)| \leq \Phi(\xi) := \E[\exp(-|\xi|^\alpha k M^{(\alpha)}_\infty)] ,
  \end{align*}
  with the parameter $k$ from \eqref{constant}, or $k=\sigma^2/2$ if $\alpha=2$.
  Notice further that, by (\ref{ligget3bis}), $\Phi$ 
  satisfies
  \begin{align}
    \label{eq.Phistationary}
    \Phi(\xi) = \E[\Phi(L\xi)\Phi(R\xi)].
  \end{align}
  Moreover, since $M_\infty^{(\a)} \not = 0$, $\E[M_\infty^{(\a)}]=1$ and $\E[(M_\infty^{(\a)})^{\gamma/\a}]<+\infty$, 
  the function $\Phi$ is positive and strictly convex in $|\xi|^\alpha$,
  with $\Phi(\xi)=1-k|\xi|^\alpha+o( |\xi|^\alpha)$.
  It follows that for each $\kappa>0$ with $\kappa<k$,
  there exists exactly one point $\Xi_\kappa >0$ with $\Phi(\Xi_\kappa )+\kappa|\Xi_\kappa |^\alpha=1$,
  and $\Xi_\kappa$ decreases monotonically from $+\infty$ to zero as $\kappa$ increases from zero to $k$.

  Since $\hat q_0=\phi_0$ is the characteristic function of the initial datum,
  satisfying the condition \eqref{stabledomain},
  it follows by Theorem 2.6.5 of \cite{IbragimovLinnik1971} that
  \begin{align*}
    \hat q_0(\xi) = 1 - k|\xi|^\alpha(1-i\eta\tan(\pi\alpha/2)\operatorname{sign}\xi) + o( |\xi|^\alpha ),
  \end{align*}
  with the same $k$ as before, and $\eta$ determined by \eqref{constant}. 
  For $\a=2$, clearly $\hat q_0(\xi)= 1- \s^2\xi^2/2 + o( \xi^2 )$.
  By the aforementioned properties of $\Phi$, 
  there exists a $\kappa\in(0,k)$ such that
  \begin{align}
    \label{eq.headstart}
    |\hat q_0(\xi)| \leq \Phi(\xi) + \kappa |\xi|^\alpha 
  \end{align}
  for all $\xi\in\RE$.
  This is evident, since for small $\xi$,
  \begin{align*}
    |\hat q_0(\xi)|=|\phi_0(\xi)| = 1 - k|\xi|^\alpha + o( |\xi|^\alpha ),
  \end{align*}
  while inequality \eqref{eq.headstart} is trivially satisfied for $|\xi|\geq\Xi_k$,
  since $|\phi_0|\leq1$.

  Starting from \eqref{eq.headstart},
  we shall now prove inductively that
  \begin{align}
    \label{eq.headinduce}
    |\hat q_\ell(\xi)| \leq \Phi(\xi) + \kappa |\xi|^\alpha .
  \end{align}
  Fix $n\geq0$, and assume \eqref{eq.headinduce} holds for all $\ell\leq n$. 
  Choose $j\leq n$.
  Using the invariance property \eqref{eq.Phistationary} of $\Phi$,
  as well as the uniform bound of characteristic functions by one,
  it easily follows that
  \begin{align*}
    |\widehat{Q}^+[\qj,\qnj](\xi)| - \Phi(\xi)
    & \leq  \E\big[ |\qj(L\xi)| |\qnj(R\xi)| - \Phi(L\xi)\Phi(R\xi) \big] \\
    & \leq \E\big[ \big( |\qj(L\xi)| - \Phi(L\xi) \big) |\qnj(R\xi)| \big] 
    + \E\big[ \Phi(L\xi) \big( |\qnj(R\xi)| - \Phi(R\xi) \big) \big] \\
    & \leq \E\big[\kappa (L|\xi|)^\alpha\big] + \E\big[\kappa(R|\xi|)^\alpha\big] = \kappa |\xi|^\alpha .
  \end{align*}
  The final equality is a consequence of $\E[L^\alpha+R^\alpha]=1$.
  By \eqref{eq.wildrec}, 
  it is immediate to conclude \eqref{eq.headinduce} with $\ell=n+1$.

  The proof is finished by noting that, since $\kappa<k$,
  \begin{align*}
    \big(1+\theta |\xi|^\alpha\big)^{-1/r} \geq \Phi(\xi) + \kappa |\xi|^\alpha 
  \end{align*}
  holds for $|\xi|\leq\rho$, 
  provided that $\rho>0$ and $\theta>0$ are sufficiently small. 
\end{proof}
\begin{lemma}
  \label{lem.tail} Under the hypotheses of Theorem \ref{thm.strongnew}, 
  let $\rho>0$ be the radius introduced in Lemma \ref{lem.head} above.
  Then, there exists a constant $\lambda>0$, independent of $\ell\geq0$, 
  \begin{align}
    \label{eq.tailinduce}
    |\hat q_\ell(\xi)| \leq (1+\lambda|\xi|^r)^{-1/r} \quad \mbox{for all $|\xi|\geq\rho$}.
  \end{align}
\end{lemma}
\begin{proof}
  Since the density $f_0$ has finite Linnik--Fisher information by hypothesis (H2),
  it follows that
  \begin{align*}
    |\phi_0(\xi)| \leq \Big( \int_\RE |\zeta|^2 |\fh(\zeta)|^2\,d\zeta \Big)
    |\xi|^{-1} 
  \end{align*}
  for all $\xi\in\RE$, 
  where $h=\sqrt{f_0}$  and  $\fh$ is its Fourier transform.
  See Lemma 2.3 in \cite{CarlenGabettaToscani}. 
  For any sufficiently small $\lambda>0$, one concludes
  \begin{align}
    \label{eq.f0tail}
    |\phi_0(\xi)| \leq ( 1 + \lambda |\xi|^r)^{-1/ r}
  \end{align}
  for sufficiently large $|\xi|$.

  Next, recall that the modulus of the characteristic function 
  of a probability density is continuous and bounded away from one, 
  locally uniformly in $\xi$ on $\RE\setminus\{0\}$.
  Diminishing the $\lambda > 0 $ in \eqref{eq.f0tail} if necessary,
  this estimate actually holds for $|\xi|\geq\rho$.

  Thus, the claim \eqref{eq.tailinduce} is proven for $\ell=0$.
  To proceed by induction, 
  fix $n\geq0$ and assume that \eqref{eq.tailinduce} holds for all $\ell\leq n$.
  In the following, we shall conclude \eqref{eq.tailinduce} for $\ell=n+1$.

  Recall that $r<\alpha$ in hypothesis (H1); see Remark \ref{rem3}.
  Hence, defining
  \begin{align}
    \label{eq.rhofromlambda}
    \rho_\lambda=(\lambda/\theta)^{1/(\alpha-r)} ,
  \end{align}
  it follows that
  \begin{align*}
    (1+\theta|\xi|^\alpha)^{-1/r} \leq (1+\lambda|\xi|^r)^{-1/r}
  \end{align*}
  if $|\xi|\geq\rho_\lambda$.
  Taking into account Lemma \ref{lem.head},
  estimate \eqref{eq.f0tail} for $\ell\leq n$ extends to all $|\xi|\geq\rho_\lambda$.
  We assume $\rho_\lambda<\rho$ from now on,
  which is equivalent to saying that $0<\lambda<\lambda_0:=\theta\rho^{\alpha-r}$.

  With these notations at hand, 
  introduce the following ``good'' set:
  \begin{align*}
    M^G_{\lambda,\delta} := \big\{ \omega: L^r(\omega)+R^r(\omega) \geq 1+\delta^r \mbox{ and } 
    \min(L(\omega),R(\omega))\rho\geq\rho_\lambda \big\} ,
  \end{align*}
  depending on $\lambda$ and a parameter $\delta>0$.
  We are going to show that if $\delta>0$ and $\lambda>0$ are sufficiently small, 
  then $M^G_{\lambda,\delta}$ has positive probability.
  First observe that the law of $(L,R)$ cannot be concentrated in the two point set $\{(0,1),(1,0)\}$
  because $\QQ(\gamma)<0$ by the hypotheses of Theorem \ref{thm.strongnew}.
  Hence we can assume $P\{ L^r+R^r>1 \}>0$, possibly after diminishing $r>0$
  (recall that if (H1) holds for some $r>0$, then it also holds for all smaller $r'>0$ as well).
  Moreover, notice that $L^r+R^r>1$ and $L=0$ or $R=0$ implies $L^\a+R^\a>1$.
  But since $\E[L^\a+R^\a]=1$, it follows that $P\{ L>0, R>0, L^r+R^r>1\}>0$.
  In conclusion, the countable union of sets
  \begin{align*}
    \bigcup_{k=1}^\infty M^G_{\lambda_0/k,1/k} 
    = \big\{\omega : L^r(\omega)+R^r(\omega) > 1, \, L(\omega)>0,\,R(\omega)>0 \big\}
  \end{align*}
  has positive probability, and so has one of the components $M^G_{\lambda_0/k,1/k}$.

  Also, we introduce a ``bad'' set, that depends on $\lambda$ and $\xi$,
  \begin{align*}
    M^B_{\lambda,\xi} := \big\{ \omega: \min(L(\omega),R(\omega))|\xi| < \rho_\lambda \big\} .
  \end{align*}
  Notice that $M^G_{\lambda,\delta}$ and $M^B_{\lambda,\xi}$ are disjoint provided $|\xi|\geq\rho$.
  
  We are now ready to carry out the induction proof, for a given $\lambda$ small enough.
  Fix $j\leq n$ and some $|\xi|\geq\rho$.
  We prove that 
  \begin{align}
    \label{eq.ptwinduce}
    \widehat{Q}^+[\qj,\qnj](\xi) \leq \E[ |\qj(L\xi)| |\qnj(R\xi)| ] \leq (1+\lambda|\xi|^r)^{-1/r} .
  \end{align}
  We distinguish several cases.
  If $\omega$ does {\em not} belong to the bad set $M^B_{\lambda,\xi}$,
  then $L|\xi|\geq\rho_\lambda$ and $R|\xi|\geq\rho_\lambda$ so that by induction hypothesis
  \begin{align*}
    |\qj(L\xi)| |\qnj(R\xi)| & \leq \big( (1+\lambda L^r|\xi|^r)(q+\lambda R^r|\xi|^r) \big)^{-1/r} \\ 
    & \leq \big( 1+ \lambda(L^r+R^r)|\xi|^r \big)^{-1/r} \leq (1+\lambda|\xi|^r)^{-1/r};
  \end{align*}
  indeed, recall that $L^r+R^r\geq1$ because of (H1).
  In particular, if $\omega$ belongs to the good set $M^G_{\lambda,\delta}$, 
  then the previous estimate improves as follows,
  \begin{align*}
    |\qj(L\xi)| |\qnj(R\xi)| 
    \leq \big( 1+\lambda(1+\delta^r)|\xi|^r \big)^{-1/r} 
    \leq \Big( \frac{1+\lambda\rho^r}{1+\lambda(1+\delta^r)\rho^r} \Big)^{1/r} ( 1+\lambda|\xi|^r )^{-1/r} ,
  \end{align*}
  where we have used that $|\xi|\geq\rho$.
  Notice further that there exists some $c>0$ --- 
  depending on $\delta$, $\theta$, $\lm_0$, $\rho$ and $r$, but not on $\lambda$ ---
  such that for all sufficiently small $\lambda>0$,
  \begin{align*}
    \Big( \frac{1+\lambda\rho^r}{1+\lambda(1+\delta^r)\rho^r} \Big)^{1/r} \leq 1 - c\lambda .
  \end{align*}
  Finally, suppose that $\omega$ is a point in the bad set $M^B_{\lambda,\xi}$,
  and assume without loss of generality that $L\geq R$.
  Then $L^r|\xi|^r\geq(1-R^r)|\xi|^r\geq|\xi|^r-\rho_\lambda^r$, and so, 
  for sufficiently small $\lambda$ and for any $\xi \geq \rho$,  
  \begin{align*}
    |\qj(L\xi)| |\qnj(R\xi)|
    \leq (1+\lambda L^r|\xi|^r )^{-1/r} \leq ( 1+\lambda|\xi|^r -\lambda\rho_\lambda^r )^{-1/r} 
    \leq (1+\lambda\rho_\lambda^r)^{1/r} (1+\lambda|\xi|^r)^{-1/r} .
  \end{align*}
  Again, there exists a $\lambda$-independent constant $C$ such that,
  for all sufficiently small $\lambda>0$,
  \begin{align*}
    (1+\lambda\rho_\lambda^r)^{1/r} \leq 1 + C\lambda\rho_\lambda^r .
  \end{align*}
  Putting the estimates obtained in the three cases together,
  one obtains
  \begin{align*}
    & \E[ |\qj(L\xi)| |\qnj(R\xi)| ] \\
    & \quad \leq (1+\lambda|\xi|^r)^{-1/r} \big[ \big(1-P(M^G_{\lambda,\delta})-P(M^B_{\lambda,\xi})\big) + P(M^G_{\lambda,\delta})(1-c\lambda) + P(M^B_{\lambda,\xi})(1+C\lambda\rho_\lambda^r) \big] \\
    & \quad \leq (1+\lambda|\xi|^r)^{-1/r} \big[ 1 + \lambda(C\rho_\lambda^r - cP\big(M^G_{\lambda,\delta})\big) \big].
  \end{align*}
  Notice that we have used the trivial estimate $P(M^B_{\lambda,\xi})\leq1$ in the last step,
  which eliminates any dependence of the term in the square brackets on $\xi$.
  To conclude \eqref{eq.ptwinduce}, it sufficies to observe that
  as $\lambda$ decreases to zero, 
  $\rho_\lambda$ tends to zero monotonically by \eqref{eq.rhofromlambda},
  while the measure $P(M^G_{\lambda,\delta})$ is obviously non-decreasing and we have
  already proved that $P(M^G_{\lambda^*,\delta})>0$ for $\lambda^*$ and $\delta$ suitably chosen.
  Hence $C\rho_\lambda^r\leq cP(M^G_{\lambda,\delta})$ when $\lambda>0$ is small enough.
  From \eqref{eq.ptwinduce}, 
  it is immediate to conclude \eqref{eq.tailinduce}, 
  recalling the recursive definition of $\hat q_{n+1}$ in \eqref{eq.wildrec}.

  Thus, the induction is complete, and so is the proof of the lemma.
\end{proof}
\begin{proof}[Proof of Theorem \ref{thm.strongnew}]
  The key step is to prove convergence 
  of the characteristic functions $\phi(t)\to\phi_\infty$ in $L^2(\RE)$.
  To this end, observe that the uniform bound on $\qn$ obtained in Lemma \ref{lem.tail} above
  directly carries over to the Wild sum,
  \begin{align*}
    |\phi(t;\xi)| \leq e^{-t} \sum_{n=0}^\infty (1-e^{-t})^n 
  |\qn(\xi)| \leq (1+\lambda|\xi|^r)^{-1/r} \qquad (|\xi| \geq \rho).
  \end{align*}
  Moreover, since  $\lim_{t \to +\infty}\phi(t;\xi)=\phi_\infty(\xi)$   for every $\xi\in\RE$,
  also
  \begin{align*}
    |\phi_\infty(\xi)| \leq (1+\lambda|\xi|^r)^{-1/r} \qquad (|\xi| \geq \rho).
  \end{align*}
  Let $\epsilon>0$ be given.
  Then there exists a $\Xi \geq \rho $ such that
  \begin{align*}
    \int_{|\xi|\geq\Xi} |\phi(t;\xi)-\phi_\infty(\xi)|^2\,d\xi 
    &\leq 2 \int_{|\xi|\geq\Xi} \big( |\phi(t;\xi)|^2 + |\phi_\infty(\xi)|^2 \big) \,d\xi \\
    &\leq 4 \int_\Xi^\infty (1+\lambda|\xi|^r)^{-2/r}\,d\xi \leq \frac\epsilon2 .
  \end{align*}
  On the other hand, by weak convergence of $V_t$ to $V_\infty$, $\phi(t;\cdot)$
converges to $\phi_\infty$  uniformly on every compact set of $\RE$ as $t \to +\infty$, 
hence   there exists a time $T>0$ such that
  \begin{align*}
    |\phi(t;\xi)-\phi_\infty(\xi)|^2 \leq \frac{\epsilon}{4\Xi}
  \end{align*}
   for every $|\xi|\leq\Xi$ and $t\geq T$.
  In combination, it follows that
  \begin{align*}
    \| \phi(t) - \phi_\infty \|_{L^2}^2 \leq \epsilon 
  \end{align*}
  for all $t\geq T$.
  Since $\epsilon>0$ has been arbitrary,  convergence of $\phi(t)$ to $\phi_\infty$ in $L^2(\RE)$ 
   follows.
  By Plancherel's identity, this immediately implies strong convergence 
  of the densities $f(t)$ of $V_t$ to the density $f_\infty$ of $V_\infty$ in $L^2$.

  Convergence in $L^1(\RE)$ is obtained by interpolation between weak and $L^2(\RE)$ convergence.
  Let $\epsilon>0$ be given, and choose $M>0$ such that
  \begin{align*}
    \int_{|x|\geq M} f_\infty(x)\,dx < \frac\epsilon4.
  \end{align*}
  By weak convergence of $V_t$ to $V_\infty$ 
  there exists a $T>0$ such that
  \begin{align*}
    \int_{|x|\geq M} f(t;x)\,dx < \frac\epsilon2
  \end{align*}
  for all $t\geq T$.
  Now Hölder's inequality implies
  \begin{align*}
    \int_\RE |f(t;x)-f_\infty(x)|\,dx 
    &\leq (2M)^{1/2} \Big( \int_{|x| \leq M} |f(t;x)-f_\infty(x)|^2\,dx \Big)^{1/2} \\
    & \quad + \int_{|x|>M} \big(|f(t;x)|+|f_\infty(x)|\big) \,dx \\
    &< (2M)^{1/2} \| f(t)-f_\infty \|_{L^2} + \frac{3\epsilon}{4} 
  \end{align*}
  Increasing $T$ sufficiently, the last sum is less than $\epsilon$ for $t\geq T$.

  Finally, convergence in $L^p(\RE)$ with $1<p<2$ follows 
  by interpolation between convergence in $L^1(\RE)$ and in $L^2(\RE)$.
\end{proof}

 \section*{Acknowledgements} 
F.B.'s research was partially supported by Ministero 
dell'Istruzione, dell'Università e della Ricerca
(MIUR grant 2006/134526).
D.M. acknowledges support from the Italian MIUR, project
``Kinetic and hydrodynamic equations of complex collisional
systems'', and from the Deutsche Forschungsgemeinschaft,
grant JU 359/7. He thanks the Department of Mathematics of
the University of Pavia, where a part of this research has been
carried out, for the kind hospitality.

All three authors would like to thank E.Regazzini
for his time, comments and suggestions.

\end{document}